\documentclass[submission,copyright,creativecommons]{eptcs}
\providecommand{\event}{ACT 2023} 
\usepackage[utf8]{inputenc}

\usepackage[pdftex,dvipsnames]{xcolor}  
\usepackage{amsmath,amsthm,amssymb,amsfonts,graphicx}
\usepackage{stmaryrd}
\usepackage{tikz}
\usepackage{proof}
\usepackage{enumerate}
\usepackage{mathtools}
\usepackage{lscape}
\usepackage{multicol} 
\usepackage{leftidx}
\usepackage{bbm}
\usepackage{rotating}
\usepackage{extarrows}
\usepackage{graphicx}
\usepackage{placeins}
\usepackage{dsfont}
\usepackage{caption}
\providecommand{\event}{ACT 2023}

\input{tikz.tex}

\newcommand{\ox}{\otimes}

\newcommand{\todo}[1]{\textcolor{red}{#1}}

\renewcommand{\phi}{\varphi}

\newcounter{dummy} 
\numberwithin{dummy}{section}

\newtheorem{lemma}[dummy]{Lemma}  
\newtheorem{theorem}[dummy]{Theorem}
\theoremstyle{definition}
\newtheorem{definition}[dummy]{Definition}

\newtheorem{corollary}[dummy]{Corollary} 
\theoremstyle{definition}

\numberwithin{equation}{section}

\newlength{\llcfoo}



\makeatletter

\newdimen\w@dth

\def\setw@dth#1#2{\setbox\z@\hbox{\scriptsize $#1$}\w@dth=\wd\z@
\setbox\@ne\hbox{\scriptsize $#2$}\ifnum\w@dth<\wd\@ne \w@dth=\wd\@ne \fi
\advance\w@dth by 1.2em}

\def\t@^#1_#2{\allowbreak\def\n@one{#1}\def\n@two{#2}\mathrel
{\setw@dth{#1}{#2}
\mathop{\hbox to \w@dth{\rightarrowfill}}\limits
\ifx\n@one\empty\else ^{\box\z@}\fi
\ifx\n@two\empty\else _{\box\@ne}\fi}}
\def\t@@^#1{\@ifnextchar_ {\t@^{#1}}{\t@^{#1}_{}}}

\def\t@left^#1_#2{\def\n@one{#1}\def\n@two{#2}\mathrel{\setw@dth{#1}{#2}
\mathop{\hbox to \w@dth{\leftarrowfill}}\limits
\ifx\n@one\empty\else ^{\box\z@}\fi
\ifx\n@two\empty\else _{\box\@ne}\fi}}
\def\t@@left^#1{\@ifnextchar_ {\t@left^{#1}}{\t@left^{#1}_{}}}

\def\two@^#1_#2{\def\n@one{#1}\def\n@two{#2}\mathrel{\setw@dth{#1}{#2}
\mathop{\vcenter{\hbox to \w@dth{\rightarrowfill}\kern-1.7ex
                 \hbox to \w@dth{\rightarrowfill}}%
       }\limits
\ifx\n@one\empty\else ^{\box\z@}\fi
\ifx\n@two\empty\else _{\box\@ne}\fi}}
\def\tw@@^#1{\@ifnextchar_ {\two@^{#1}}{\two@^{#1}_{}}}

\def\tofr@^#1_#2{\def\n@one{#1}\def\n@two{#2}\mathrel{\setw@dth{#1}{#2}
\mathop{\vcenter{\hbox to \w@dth{\rightarrowfill}\kern-1.7ex
                 \hbox to \w@dth{\leftarrowfill}}%
       }\limits
\ifx\n@one\empty\else ^{\box\z@}\fi
\ifx\n@two\empty\else _{\box\@ne}\fi}}
\def\t@fr@^#1{\@ifnextchar_ {\tofr@^{#1}}{\tofr@^{#1}_{}}}

\newdimen\W@dth
\def\setW@dth#1#2{\setbox\z@\hbox{$#1$}\W@dth=\wd\z@
\setbox\@ne\hbox{$#2$}\ifnum\W@dth<\wd\@ne \W@dth=\wd\@ne \fi
\advance\W@dth by 1.2em}

\def\T@^#1_#2{\allowbreak\def\N@one{#1}\def\N@two{#2}\mathrel
{\setW@dth{#1}{#2}
\mathop{\hbox to \W@dth{\rightarrowfill}}\limits
\ifx\N@one\empty\else ^{\box\z@}\fi
\ifx\N@two\empty\else _{\box\@ne}\fi}}
\def\T@@^#1{\@ifnextchar_ {\T@^{#1}}{\T@^{#1}_{}}}

\def\T@left^#1_#2{\def\N@one{#1}\def\N@two{#2}\mathrel{\setW@dth{#1}{#2}
\mathop{\hbox to \W@dth{\leftarrowfill}}\limits
\ifx\N@one\empty\else ^{\box\z@}\fi
\ifx\N@two\empty\else _{\box\@ne}\fi}}
\def\T@@left^#1{\@ifnextchar_ {\T@left^{#1}}{\T@left^{#1}_{}}}

\def\Tofr@^#1_#2{\def\N@one{#1}\def\N@two{#2}\mathrel{\setW@dth{#1}{#2}
\mathop{\vcenter{\hbox to \W@dth{\rightarrowfill}\kern-1.7ex
                 \hbox to \W@dth{\leftarrowfill}}%
       }\limits
\ifx\N@one\empty\else ^{\box\z@}\fi
\ifx\N@two\empty\else _{\box\@ne}\fi}}
\def\T@fr@^#1{\@ifnextchar_ {\Tofr@^{#1}}{\Tofr@^{#1}_{}}}

\def\Two@^#1_#2{\def\N@one{#1}\def\N@two{#2}\mathrel{\setW@dth{#1}{#2}
\mathop{\vcenter{\hbox to \W@dth{\rightarrowfill}\kern-1.7ex
                 \hbox to \W@dth{\rightarrowfill}}%
       }\limits
\ifx\N@one\empty\else ^{\box\z@}\fi
\ifx\N@two\empty\else _{\box\@ne}\fi}}
\def\Tw@@^#1{\@ifnextchar_ {\Two@^{#1}}{\Two@^{#1}_{}}}

\def\to{\@ifnextchar^ {\t@@}{\t@@^{}}}
\def\from{\@ifnextchar^ {\t@@left}{\t@@left^{}}}
\def\tofro{\@ifnextchar^ {\t@fr@}{\t@fr@^{}}}
\def\To{\@ifnextchar^ {\T@@}{\T@@^{}}}
\def\From{\@ifnextchar^ {\T@@left}{\T@@left^{}}}
\def\Two{\@ifnextchar^ {\Tw@@}{\Tw@@^{}}}
\def\Tofro{\@ifnextchar^ {\T@fr@}{\T@fr@^{}}}

\makeatother
\newcommand{\vcenteredinclude}[2]{\begingroup
\setbox0=\hbox{\includegraphics[#1]{#2}}%
\parbox{\wd0}{\box0}\endgroup}

\begin{document}
\title{Normalizing Resistor Networks}
\author{Robin Cockett
\institute{Department of Computer Science}
\institute{University of Calgary\\
Alberta, Canada}
\email{robin@ucalgary.ca}
\and
Amolak Ratan Kalra 
\institute{Institute for Quantum Computing (IQC)\\ Cheriton School of Computer Science}
\institute{University of Waterloo\\
Ontario, Canada}
\email{arkalra@uwaterloo.ca}
\and
Priyaa Varshinee Srinivasan
\institute{University of Calgary\\
Alberta, Canada\\ National Institute of Standards and Technology,\\
Maryland, USA}
\email{priyaavarshinee@gmail.com}
}

\def\titlerunning{Normalizing Resistor Networks}
\def\authorrunning{R.Cockett, A.R Kalra, P. Srinivasan}

\maketitle 
\begin{abstract}
Star to mesh transformations are well-known in electrical engineering, and are reminiscent of local complementation for graph states in qudit stabilizer quantum mechanics. This paper describes a rewriting system for resistor circuits over any positive division rig using general star to mesh transformations. We show how these transformations can be organized into a confluent and terminating rewriting system on the category of resistor circuits. Furthermore, based on the recently established connections between quantum and electrical circuits, this paper pushes forward the quest for approachable normal forms for stabilizer quantum circuits.  
\end{abstract}

\section{Introduction}

Electrical circuits are well-studied and, indeed, the basis of an eponymous engineering discipline. One would therefore expect that there is not much more that can be usefully said about the simplest and most basic of these circuits, namely circuits consisting of just resistors.   However, it turns out that there is always more to say!  Indeed, it seems possible that modern mathematical methods can even provide new insight into what is an old and well-studied subject. Furthermore, 
by considering resistor networks with resistance values in finite fields -- which is not the most natural direction of generalization from an electrical engineering perspective -- has a tantalizing connection to the theory of qudit stabilizer quantum mechanics.

 A categorical description for electrical circuits was provided in Brendan Fong's thesis \cite{fong2016algebra}, described in a paper with John Baez \cite{baez2015compositional}, and was also the subject of Brandon Coya's thesis \cite{Coyer}.  Following the work of Cole Comfort and Alex Kissinger \cite{comfort2021graphical}, the current authors with Shiroman Prakash investigated the relationship between electrical circuits and quantum circuits \cite{cockett2022categories}. There it was noted that the structure of parity check matrices arising from resistor networks and graph states are precisely the same. This work, in turn, relied on the developments of Graphical Linear Algebra \cite{zanasi} where it was realized that there was a natural encoding of resistors (and electrical circuits) into categories of linear relations.

Resistor networks (or circuits) form a hypergraph category, \cite{fong2019hypergraph}, which we call ${\sf Resist}$: this is a symmetric monoidal category in which each object is a commutative Fr\"obenius algebra (coherent with the tensor product).  It is an open issue as to whether the equality of maps in ${\sf Resist}$ can in general be resolved by a simple rewriting system \cite{Baeztalk}.
We resolve this question in this article for resistor circuits over a positive division rig. Our rewriting system uses an important identity for electrical circuits of resistors called the star/mesh or $(Y/\Delta)$ identity, which asserts that a ``star-shaped'' circuit (with $n$-points) is equivalent to a ``mesh-shaped'' circuit (on $n$ points).  This is a classical observation in electrical engineering with proofs going back almost a century \cite{stardelta}.  

While it is well-known that the $(Y/\Delta)_3$ transform for three nodes is a two-way identity, in the sense that any three pointed star can be transformed into a triangular mesh and conversely any such mesh can be transformed into a star, this fails for $n > 3$.  It fails for a simple reason: meshes of resistors with $n$-nodes when $n>3$ have more degrees of freedom than stars with $n$-nodes.  Meshes on $n$-nodes have $n(n-1)/2$ resistors while stars have only $n$: only at $n=3$ do they have the same number of resistors!  Thus, it is not the case that {\em every\/} $n$-mesh is be equivalent to a $n$-star for $n>3$.  However, it does remain the case that (for every $n \in \mathbb{N}$) every $n$-star can be transformed into an equivalent $n$-mesh, thereby, suggesting a natural orientation for these identities.

Not surprisingly the normalizing procedure we introduce using the general star/mesh identity is oriented in the star to mesh direction: we prove that this forms part of a confluent rewriting system on the category of resistor circuits, ${\sf Resist}$.  Regrettably, this rewriting does not make the circuits more efficient in terms of hardware real-estate, however, it certainly does provide a simple, easily automated, decision procedure for equality. The resulting normal form for circuits is a family of meshes (with ``extra inputs''); a form foreshadowed not only by the work in \cite{fong2016algebra,baez2015compositional,Coyer} but also in the work using parity matrices \cite{cockett2022categories}. 



While this paper provides a rewriting system for resistor circuits in which the conductances are taken from an arbitrary {\em positive division rig},
a more desirable objective would be to show that our results hold for {\em all\/} division rigs.  All fields, including finite fields, are examples of division rigs.  Resistors over finite fields can be interpreted as (special) stabilizer quantum circuits \cite{kalra2022category}: thus, obtaining such a generalization would provide normal forms for these quantum circuits which in turn could possibly be generalized to arbitrary stabilizer circuits.

We choose positive rigs in our formulation because the normalization procedure demand division by sums, hence,  in a non-positive rig one can possibly run into a division by zero situation during rewriting. Clearly, $\mathbb{R}_{>0}$, the usual ``base'' for resistors in electrical engineering is a positive division rig. For our normalization procedure, we have also  chosen to using conductances rather than impedances (or resistances) of resistors in an attempt to simplify the calculations. However, the calculation in its impedance form may also provide some advantages as there is only one case in which division by a sum may occur --- during a parallel rewrite.   Unfortunately, simply switching to impedance form does not quite suffice to allow the generalization to arbitrary division rigs!

Reducing resistor networks using series, parallel, and (Y-$\Delta$) transformations, or a combination of these to eliminate the internal nodes of the network is a well-studied problem in electrical engineering. Indeed, there exist several studies on general methods and efficient algorithms to automate the reduction process of large resistor networks \cite{Romm09,PhysRevB.37.302,Geo88,Knud06}. However, to the best of our knowledge, a terminating confluent rewriting system for such networks is yet to be found and have been suggested not to exist due to the inherent directionality of the reduction rules \cite{Baeztalk}. The main contributions of this paper are a categorical presentation of resistor networks as hypergraph categories, and a terminating confluent rewriting system for these networks based on star-mesh transformations.

\medskip 

\noindent
{\bf Notation:} Throughout this paper composition is written in diagrammatic order: $fg$ means apply $f$ followed by $g$. The string diagrams are to be read from top to bottom (following the direction of gravity) or left to right.

\section{Background}
\subsection{Spider Rewriting}

A category is a {\bf hypergraph category} in case it is a symmetric monoidal category in  which every object is coherently a special commutative Fr\"obenius algebra, \cite{fong2019hypergraph}, see Appendix \ref{Appendix: hypergraph} for details. There is a well-known rewriting system on any hypergraph category, often called ``spider'' rewriting, which normalizes  Fr\"obenius operations,  $\circ^m_n$, called ``spiders'' which have $m$ inputs and $n$ outputs (whose order does not matter).  Within this rewriting system, two spiders which are connected can then be amalgamated to form a bigger spider, see equation {\sf Spider}-(a).  The ``special'' rule allows loops, to be eliminated, see equation {\sf Spider}-(b). The main rewriting rules for spiders are:
\begin{align*}
   {\sf[Spider]} ~~~~~ (a)~ \vcenteredinclude{scale=0.12}{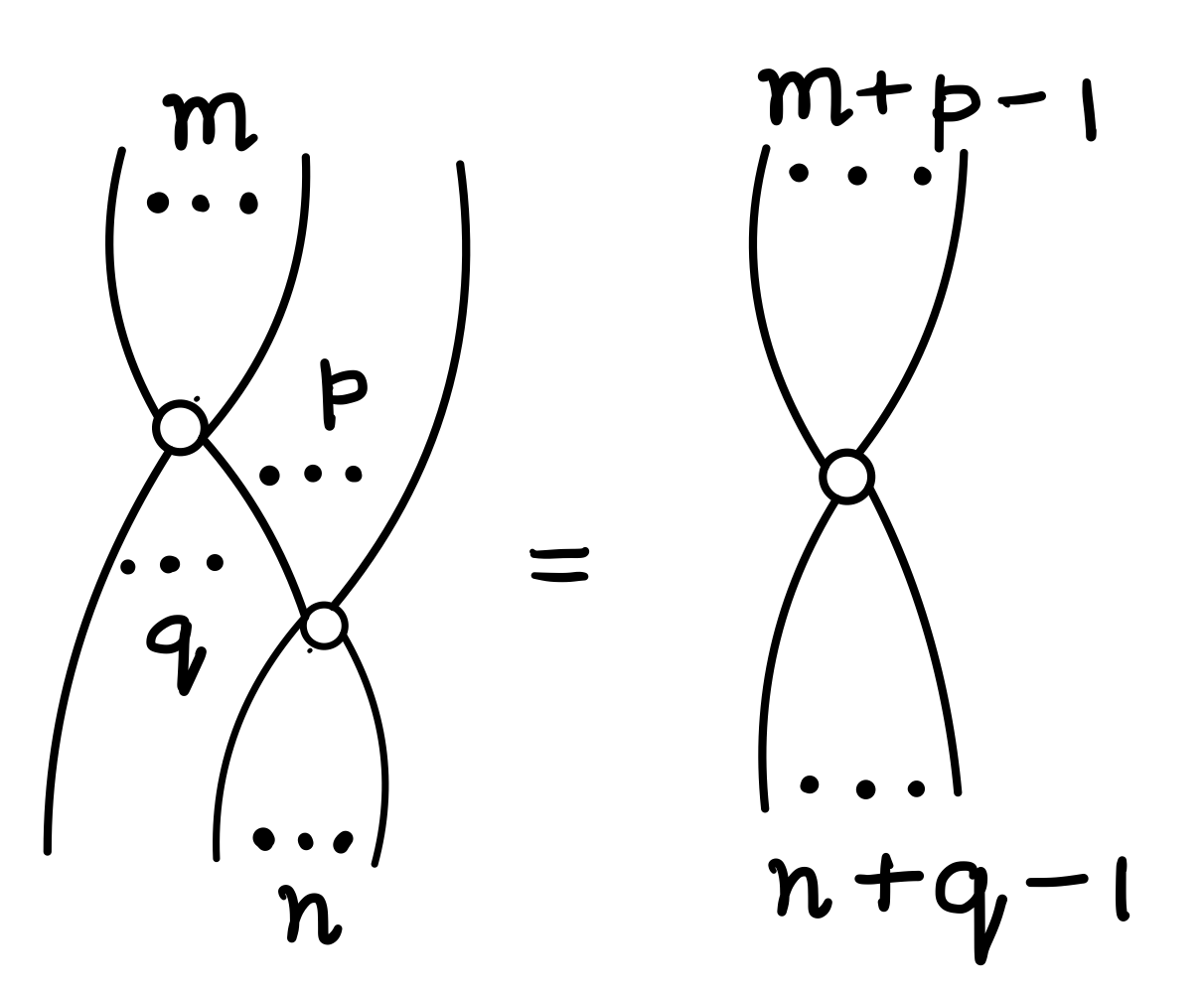} ~~~~ 
    (b)~ \vcenteredinclude{scale=0.1}{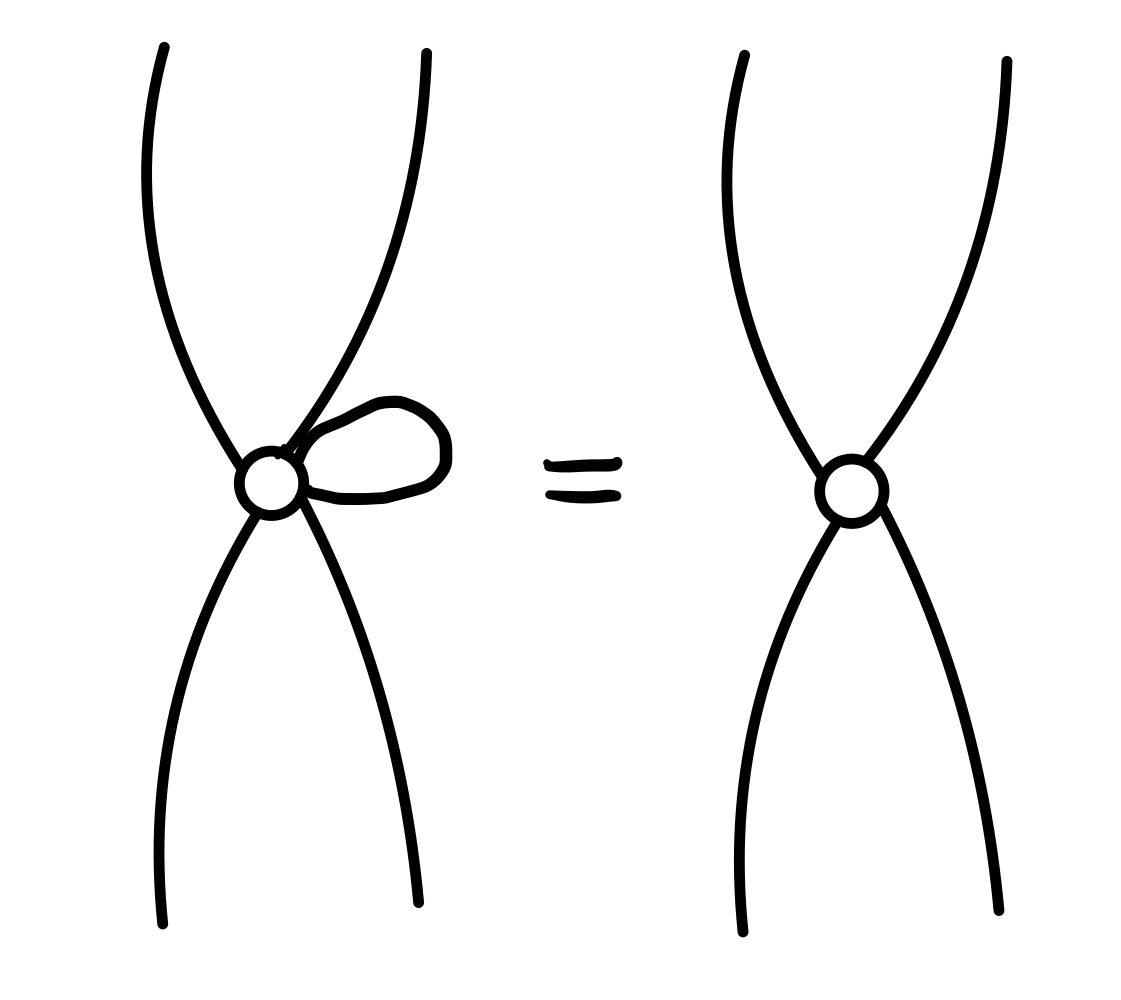} 
\end{align*}

The following unusual rewrite is an {\em expansion} which replaces a wire with a wire with a $\circ^1_1$ junction:
\begin{align*}
 {\sf[Expansion]} ~~~~~ \vcenteredinclude{scale=0.05}{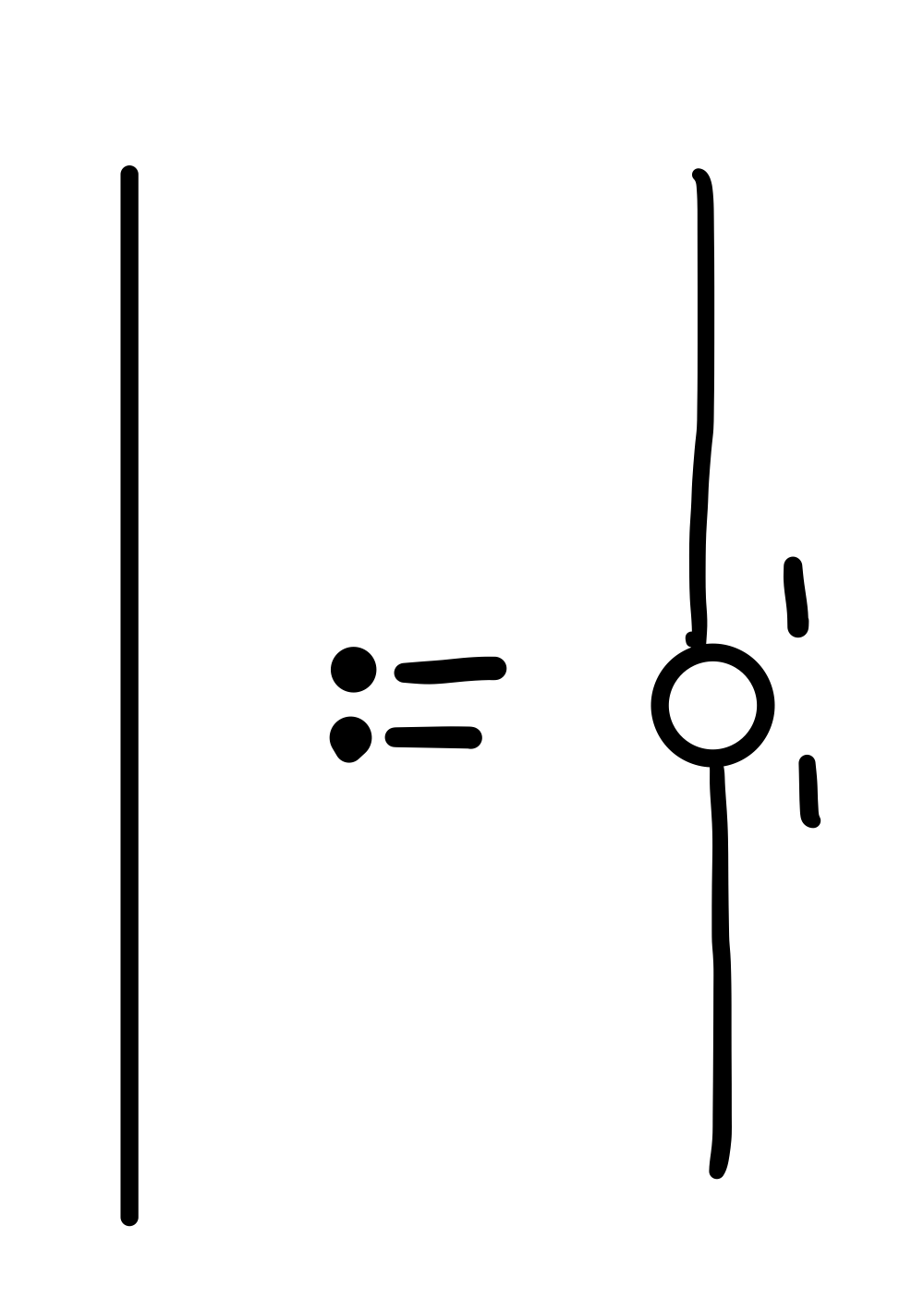}
\end{align*}
 
Clearly this, as a rewrite, can be performed indefinitely, however, an expansion is only used when there are no rewrites which can immediately undo it: thus, expansion rules are used only {\em irreducibly\/}.  For example, an expansion of a wire on which there is already a spider can always be reduced and so is not irreducible. 

The spiders $\circ^0_1$, $\circ^1_0$, and $\circ^1_1$ are given by the unit, the counit, and the identity maps respectively. The spiders $\circ^0_2$ and $\circ^2_0$ are given as follows:
\[ \vcenteredinclude{scale=0.06}{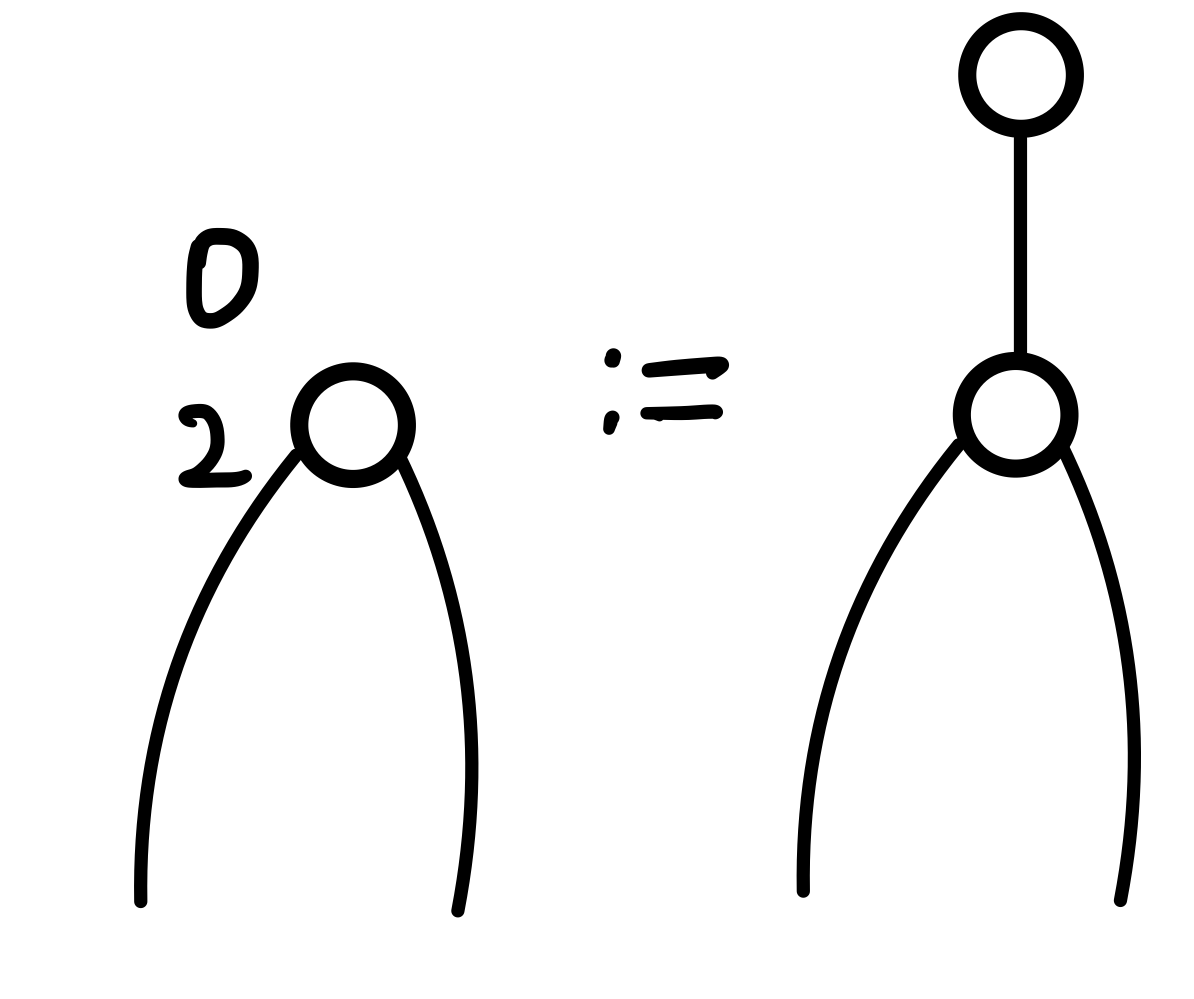} ~~~~~~~~
   \vcenteredinclude{scale=0.06}{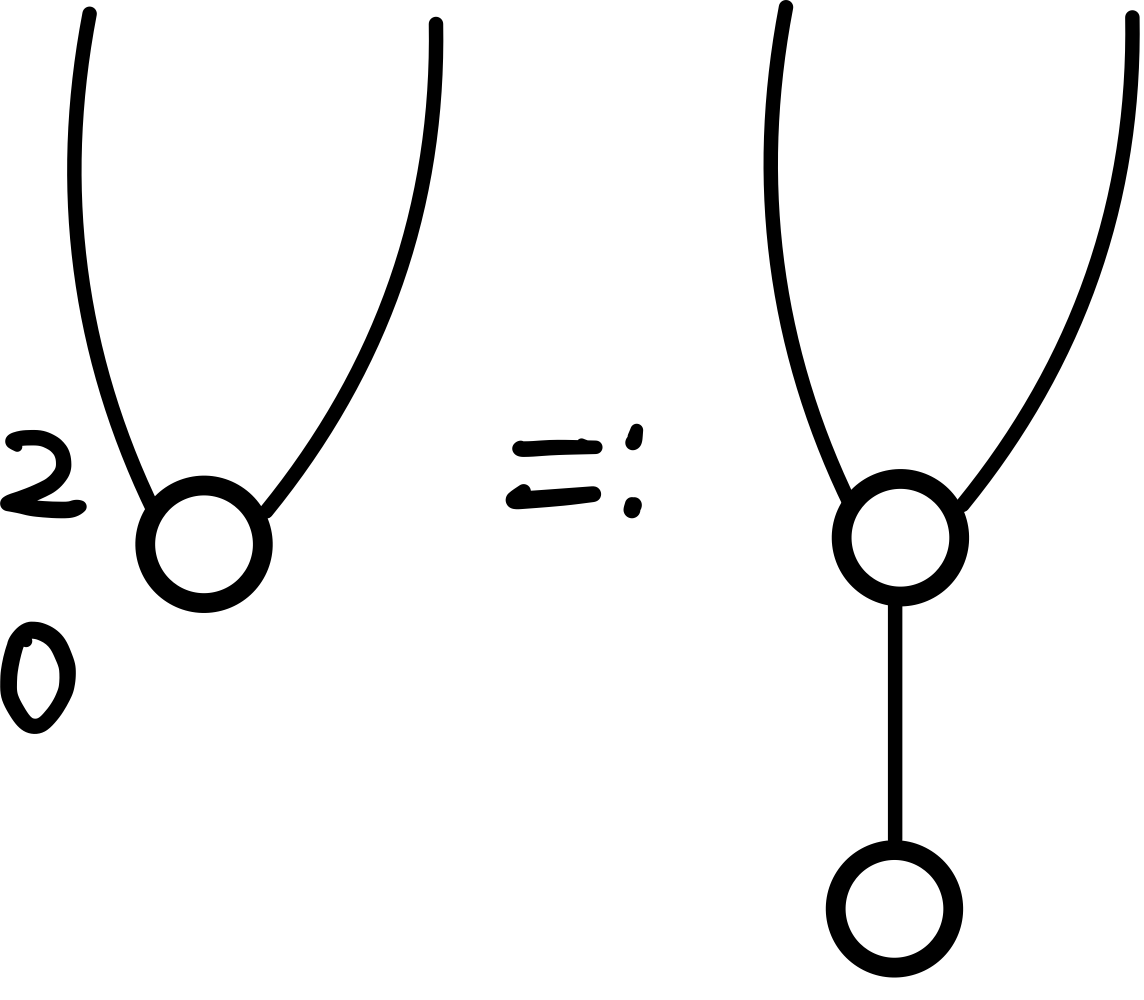}  \]

A few examples of the spider rewriting are as follows:
   \[ (a) \vcenteredinclude{scale=0.1}{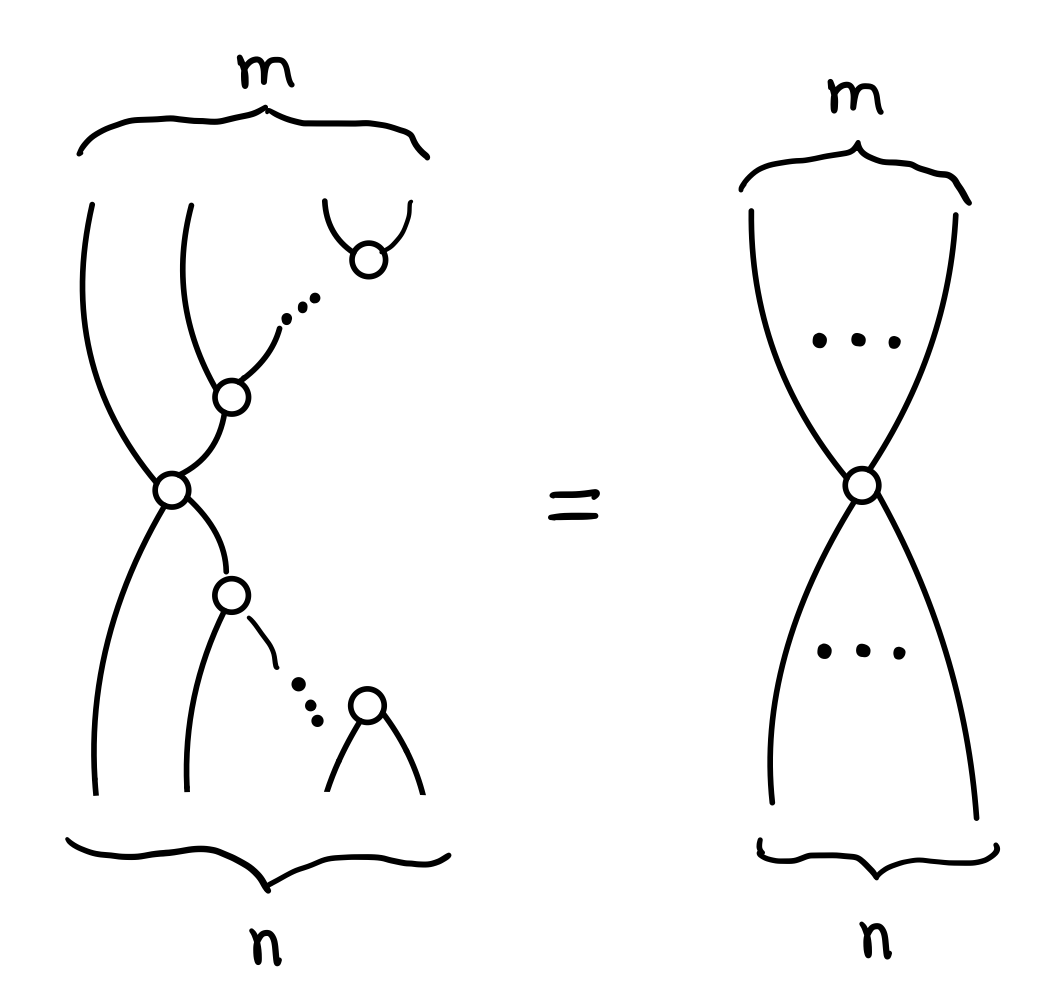} ~~~~~~~~
     (b) \vcenteredinclude{scale=0.07}{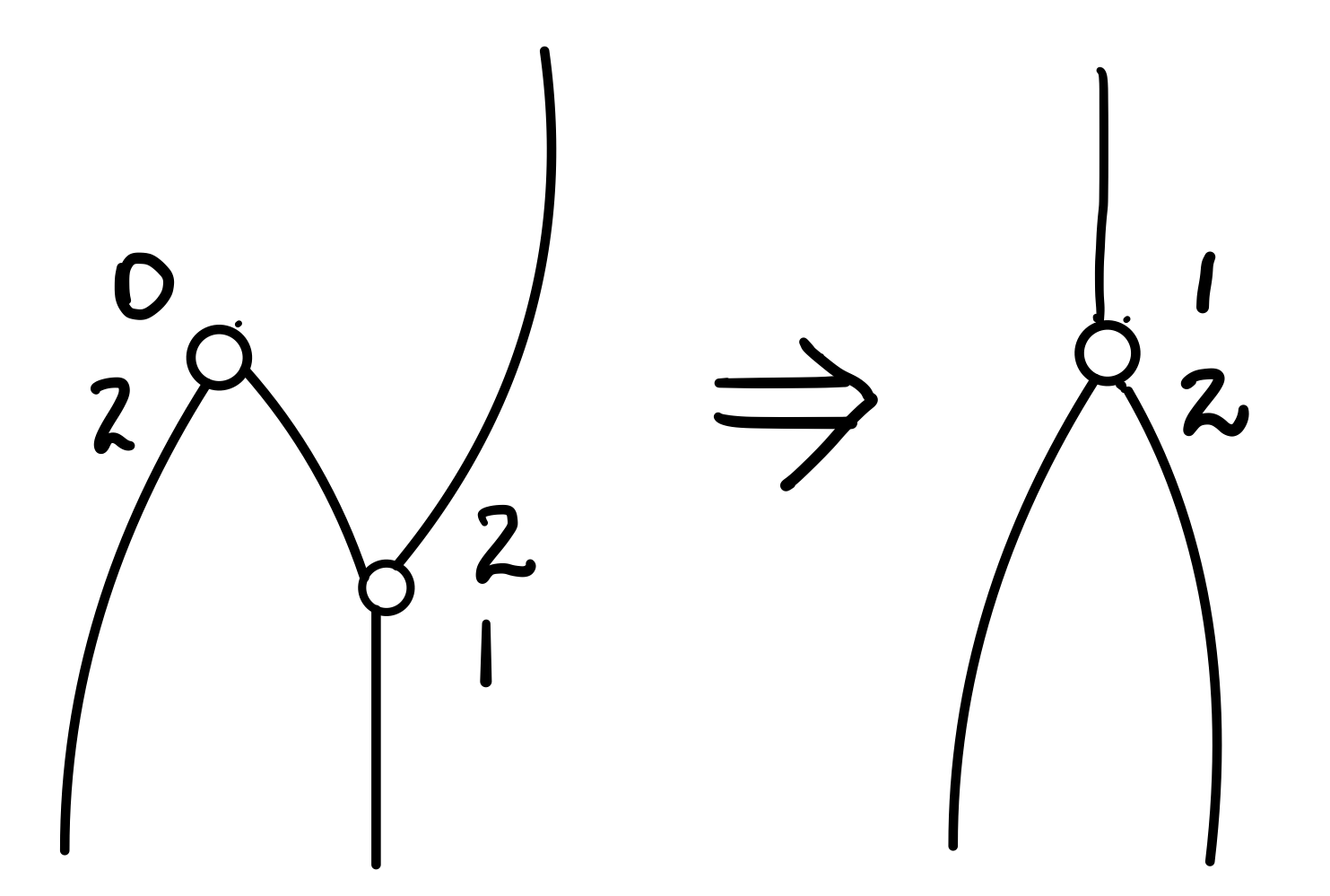}\]

\subsection{Rigs}

The category ${\sf Resist}_R$ is built atop a rig, $R$, which must satisfy some special properties: the elements of the rig will represent conductances (recall we shall work with conductances rather than impedances) as this makes for slightly simpler calculations.  This means that the rule for amalgamating parallel resistors is achieved by simply adding their conductances.  However, composing resistors in series then becomes more complicated and uses the ability to ``divide''. 

Recall that a {\bf rig} is a ``ring without negatives'' in the sense that, under addition it is a commutative monoid, and under multiplication a monoid.  These operations must satisfy the distributive laws: $r \cdot (p + q) = r \cdot p + r \cdot q$,  $(p + q) \cdot r = p \cdot r + q \cdot r$, and $0 \cdot r = 0 = s \cdot 0$.   Here we shall  consider only {\bf commutative} rigs in which $p \cdot q = q \cdot p$.   The paradigmatic and initial rig is the rig of natural numbers $\mathbb{N}$.   

A {\bf division rig} is a rig in which all non-zero elements have a multiplicative inverse.  Fields are clearly examples of division rigs as are the positive rationals, $\mathbb{Q}_{>0}$, and the positive reals, $\mathbb{R}_{>0}$. Furthermore, the two element lattice with join as addition and meet as multiplication is also a division rig.   

A rig is {\bf positive} if $x + y = 0$ implies that both $x=0$ and $y=0$.  $\mathbb{Q}_{\geq 0}$, $\mathbb{R}_{\geq 0}$, and the two element lattice are all positive division rigs in this sense. However, fields and, in particular, finite fields are not positive rigs.  

Another important example of a positive division rig is the so called ``tropical'' rig, $(\mathbb{R} \cup \{ -\infty\}, \vee,+)$, where the addition of the rig uses the maximum (with unit $-\infty$) and multiplication uses addition (with $-\infty$ as a zero).  

Below we show how to build a category of resistors based on a positive division rig, $R$.

\section{The category ${\sf Resist}_R$}

\begin{definition}
\label{Defn: ResistR}
The category ${\sf Resist}_R$, where $R$ is a positive division rig, is a hypergraph category that consists of:
\begin{description}
\item [Objects:] Natural numbers 
\item [Maps:] Generated by {\bf conductances} $y: 1\to 1$ where $y~\in~ R$ with $y\neq 0$ and ``junctions'' $\circ^n_m$.
\end{description}
\end{definition}
Together with the [{\sf Spider}] rules, the maps satisfy the following three identities (and a family of star-mesh identities): 
\[ 
[{ \sf Self-adjoint}]~~ \vcenteredinclude{scale=0.07}{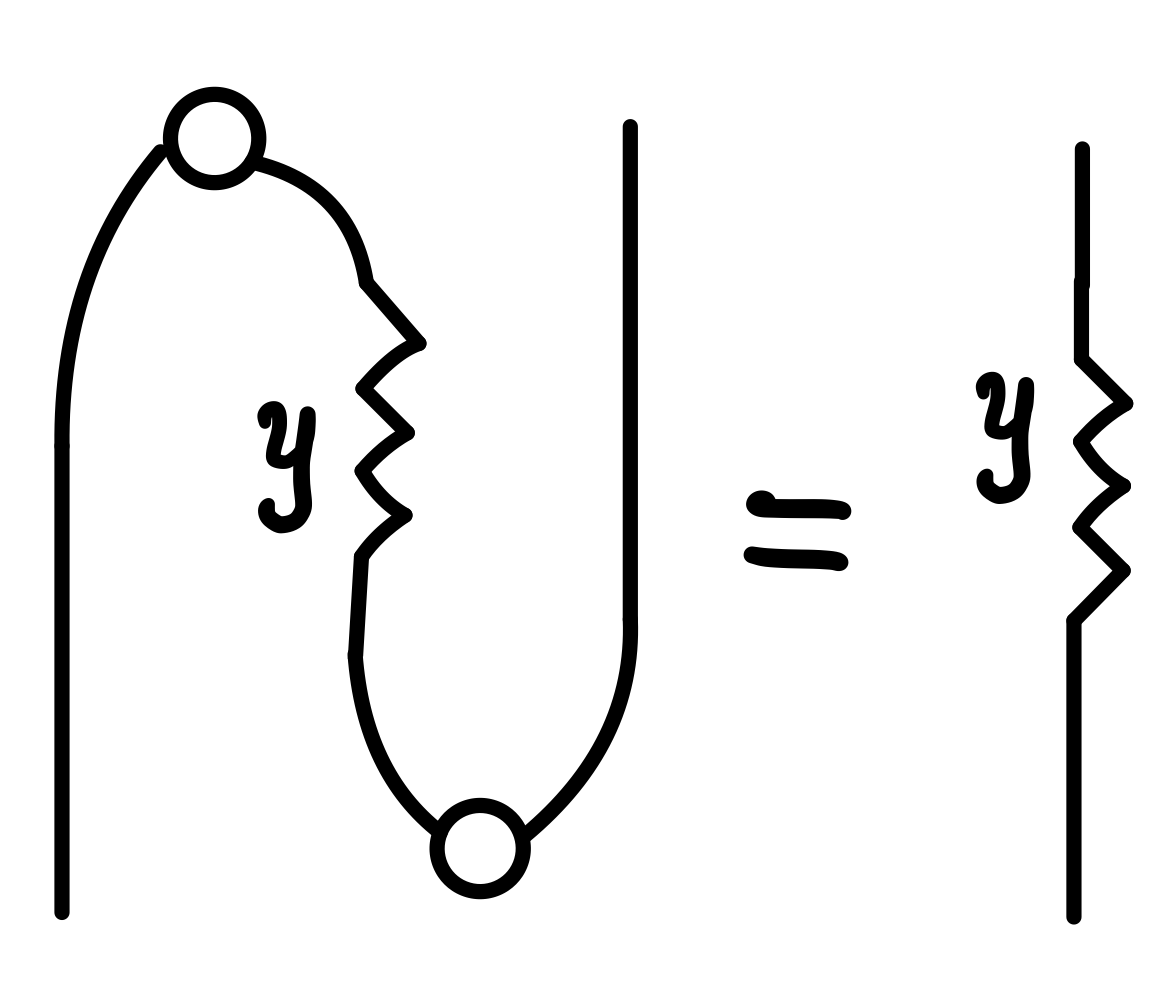} ~~~~~~
[{\sf Short~circuit}]~~ \vcenteredinclude{scale=0.07}{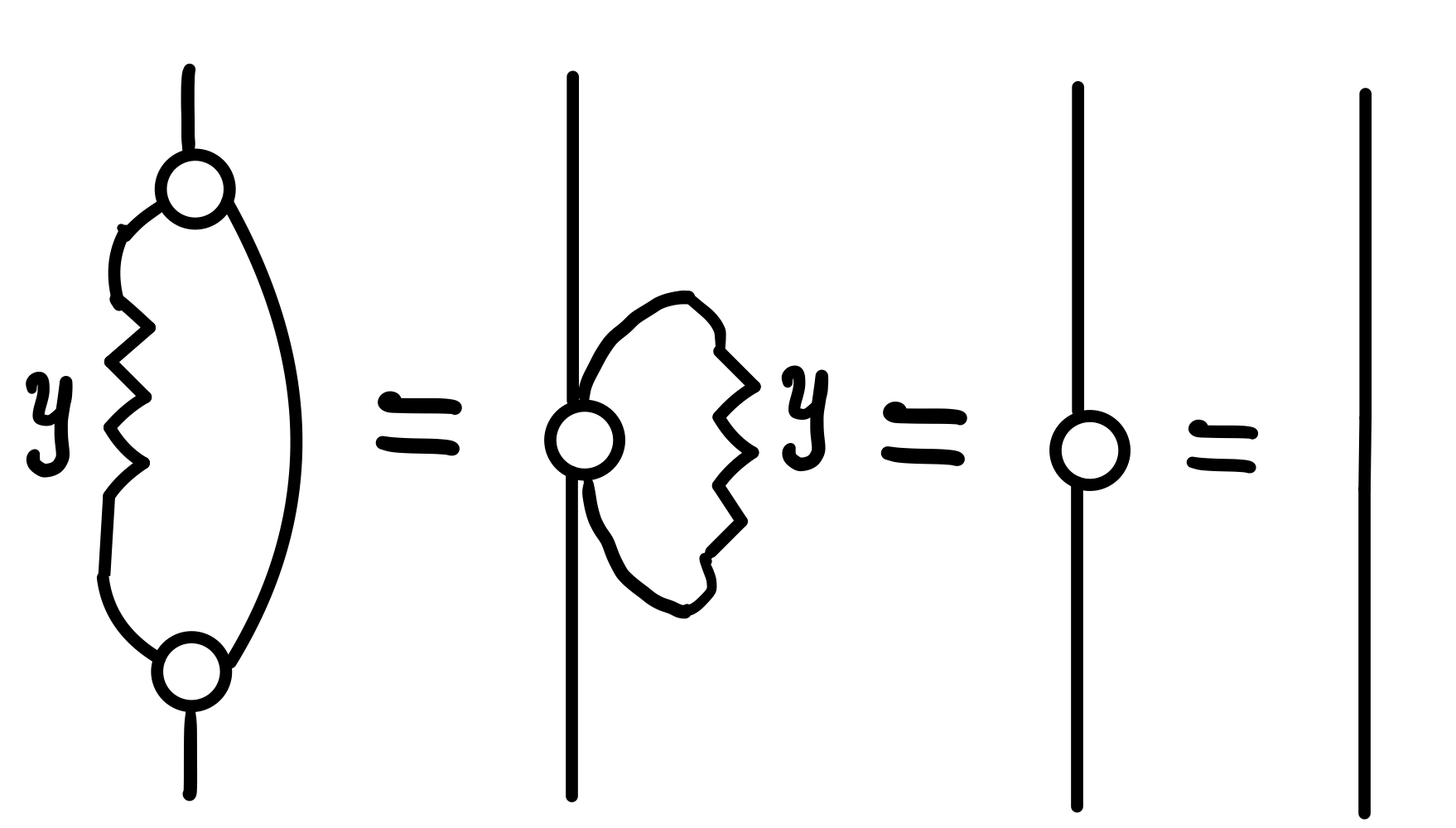} 
\]
\[
[{\sf Parallel}]~~ \vcenteredinclude{scale=0.07}{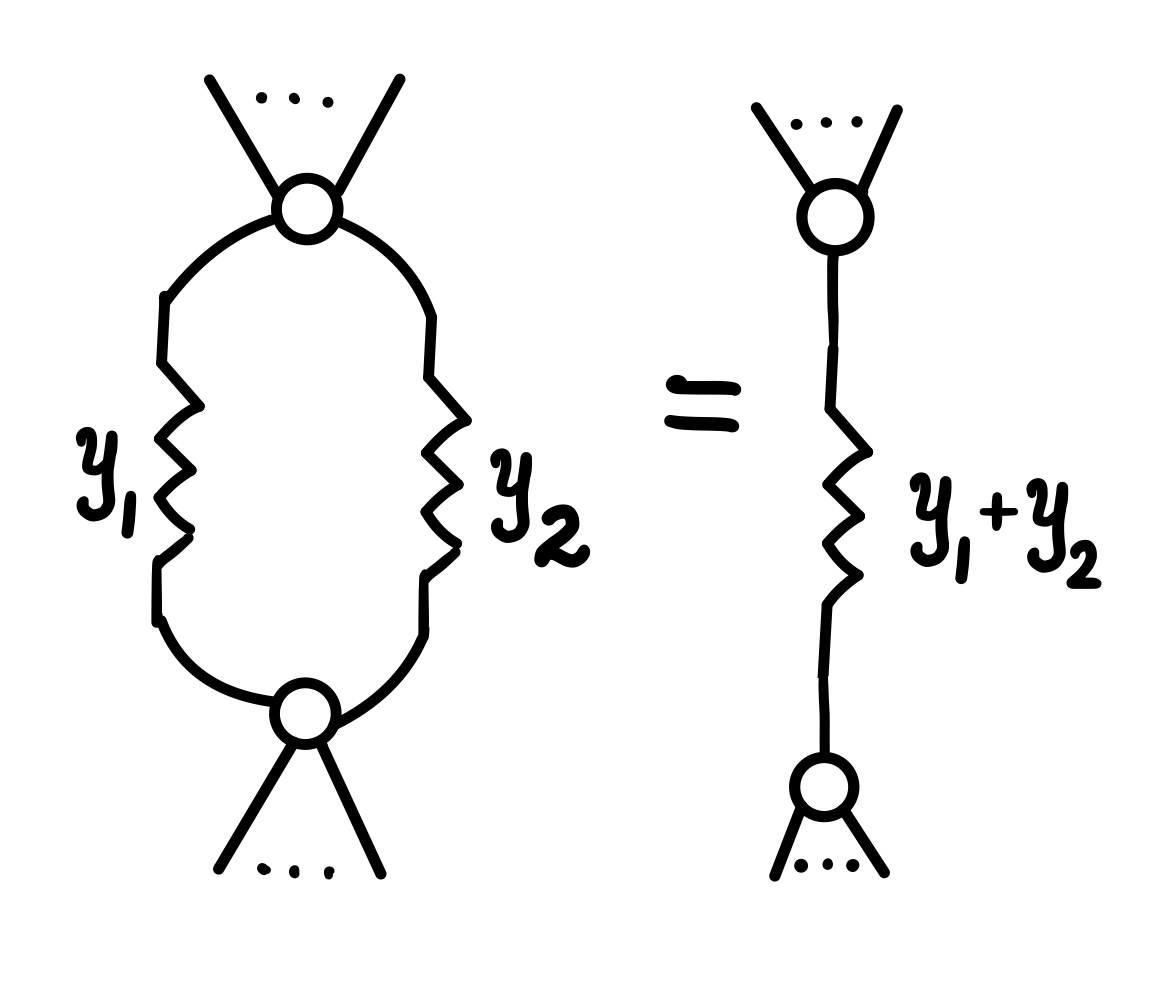}
\]

 Thus, resistors are self-adjoint. The [{\sf Short circuit}] rule states that if there is an infinite conductance in parallel with a conductance of any finite value, the current would take the path of ``least resistance" can flow through the infinite conductance wire.  The [{\sf Parallel rule}] is used to collapse a number of parallel conductance into one conductance. A wire may be thought of as a resistor with infinite conductance.  

The resistors in addition satisfy a family of {\em star-mesh identities}, $(Y/\Delta)_n$.  Each of these identities equates a star resistor network, that is a network which has one {\em internal node} (a node which has connections only within the circuit) to a completely connected graph of resistors with no internal nodes which is referred to as a {\em mesh}. 

\begin{figure}[h]
\centering
\includegraphics[scale=0.1]{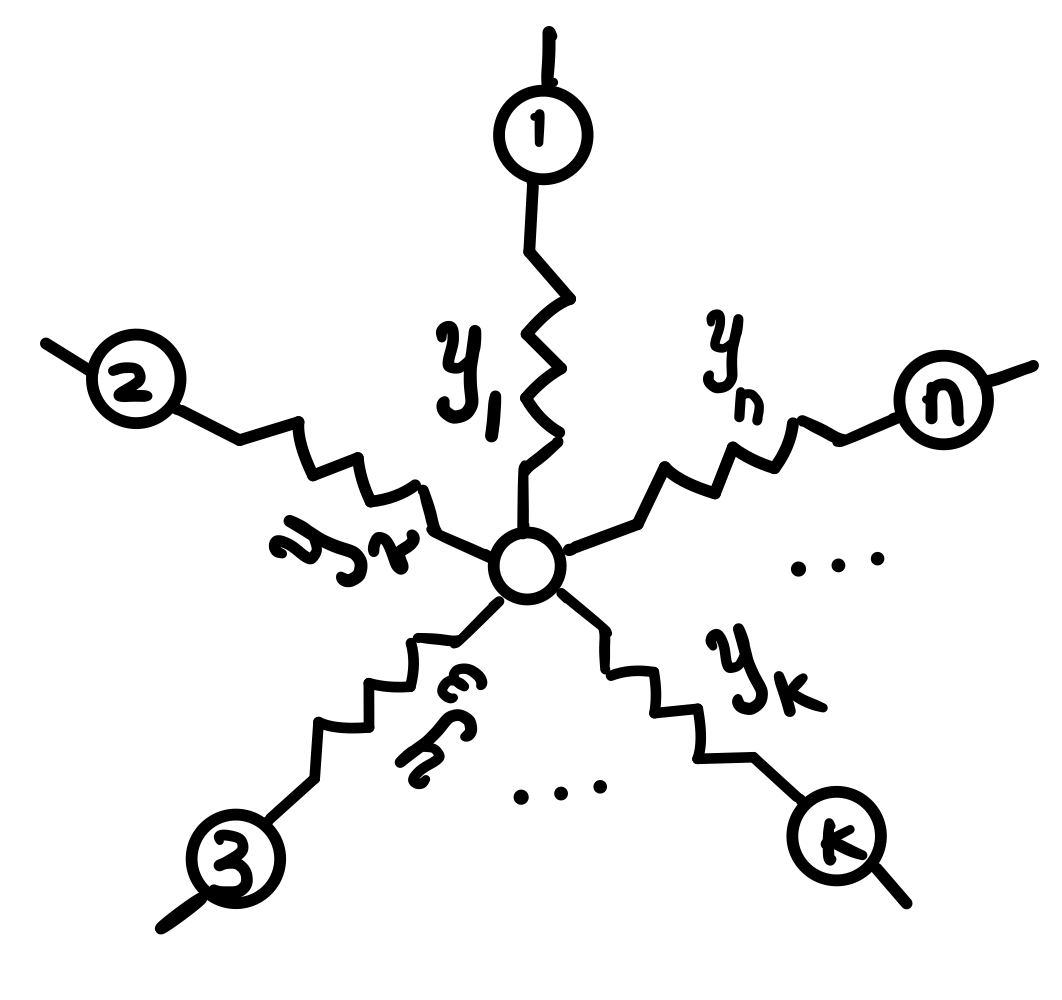}
\caption{n-node star network}
\label{Fig: star}
\end{figure}

\FloatBarrier

Given an n-node star network as shown in Figure \ref{Fig: star}, the corresponding mesh network consists of a completely connected graph with nodes $1, 2, 3, \cdots, n$ in which the edge between each pair of nodes $i$ and $j$ have conductance value,
\begin{equation}
\label{eqn: mesh conductance}
Y_{ij}= \frac{y_{i}y_{j}}{\sum_{k=1}^{n}y_{k}}
\end{equation}



A few special cases of the star-mesh transformation are shown below:
\begin{figure}[!h]
    \begin{minipage}[b]{.2\linewidth}
        \centering
        \includegraphics[scale=0.03]{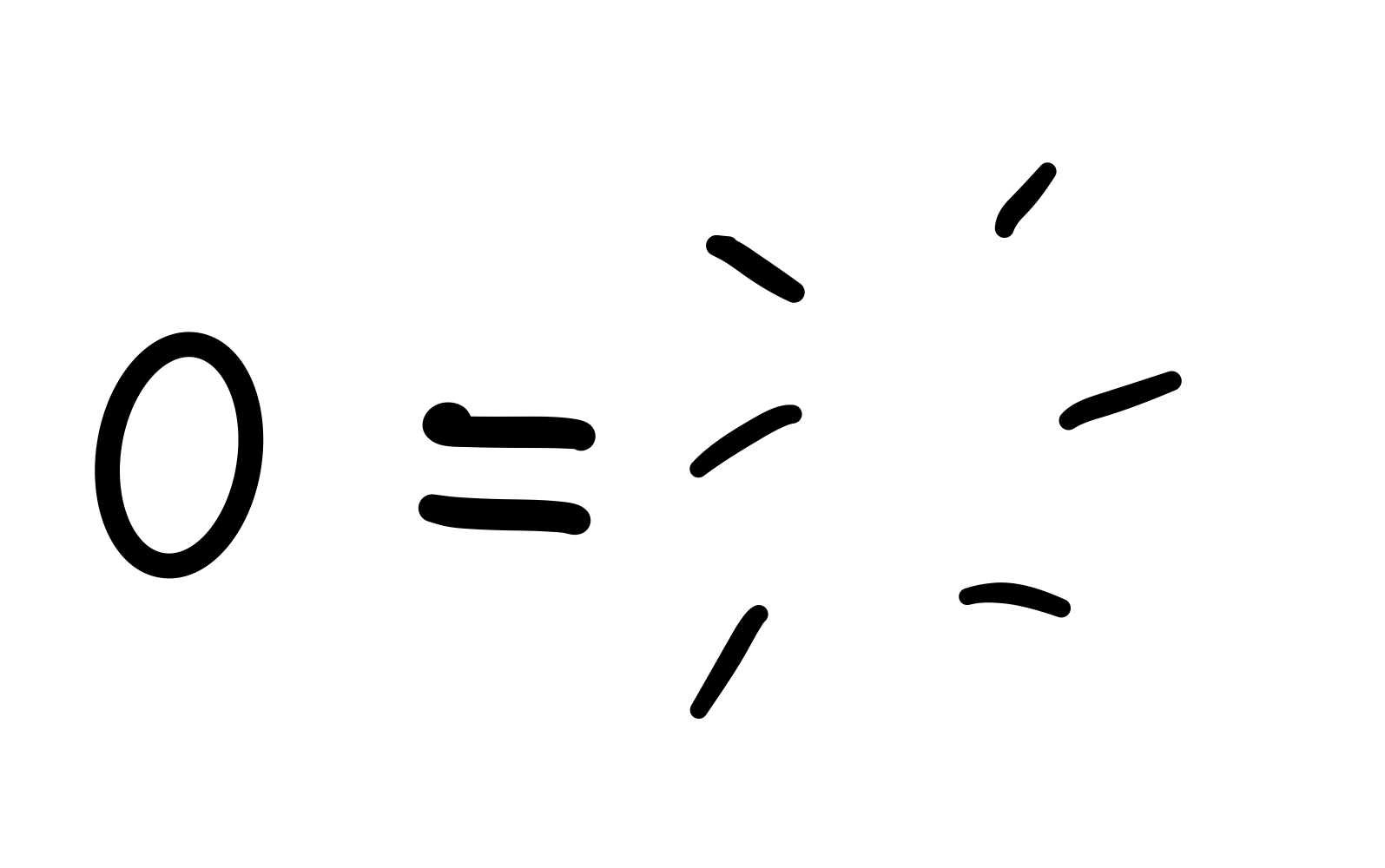}
        {\caption*{(Y/$\Delta$)$_0$}}
            \end{minipage}%
    \begin{minipage}[b]{.2\linewidth}
        \centering
        \includegraphics[scale=0.028]{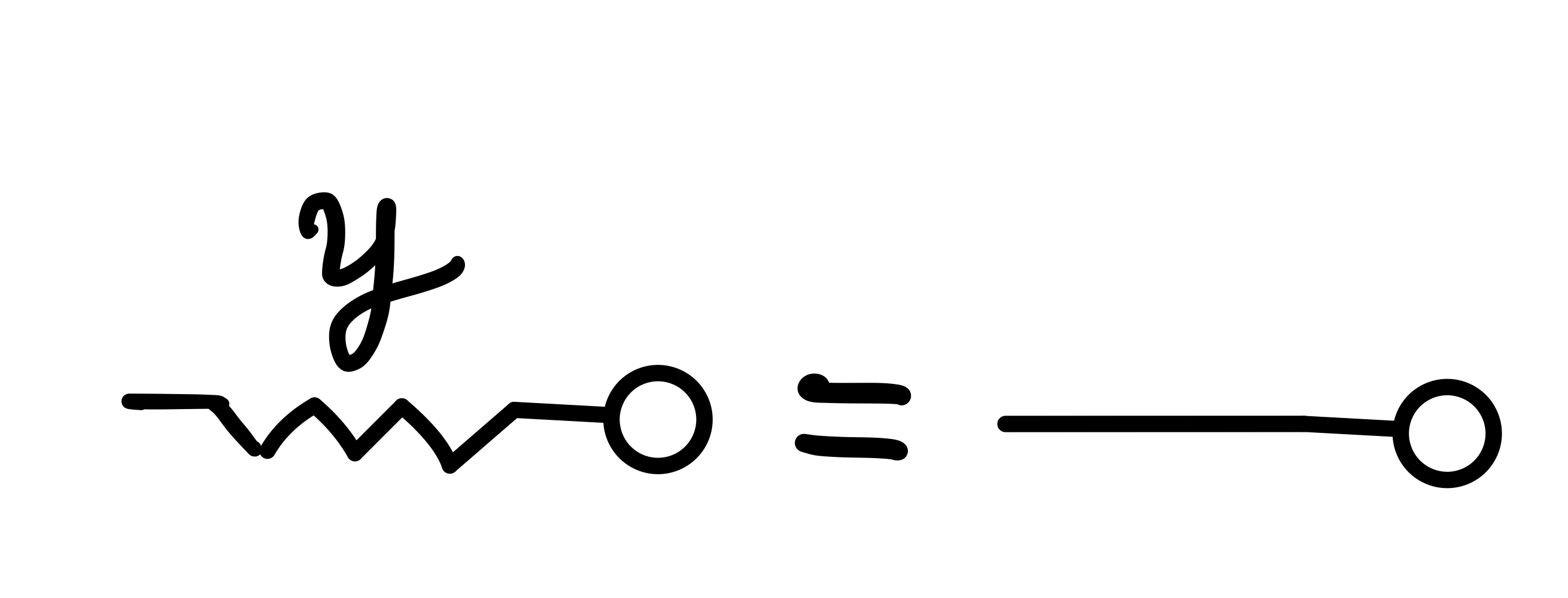}
        {\caption*{(Y/$\Delta$)$_1$}}
    \end{minipage}   
    \begin{minipage}[b]{.2\linewidth}
        \centering
        \includegraphics[scale=0.06]{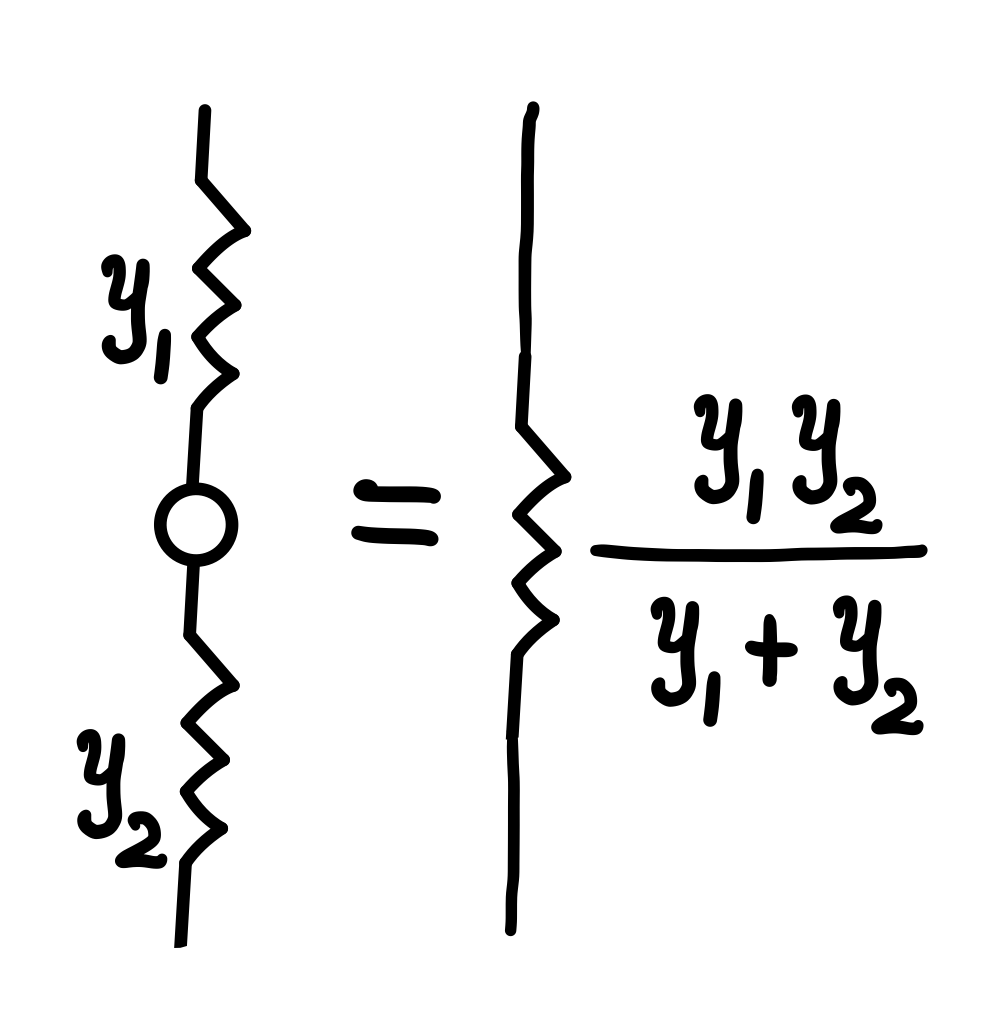}
        {\caption*{(Y/$\Delta$)$_2$}}
    \end{minipage} 
    \begin{minipage}[b]{.3\linewidth}
        \centering
        \includegraphics[scale=0.06]{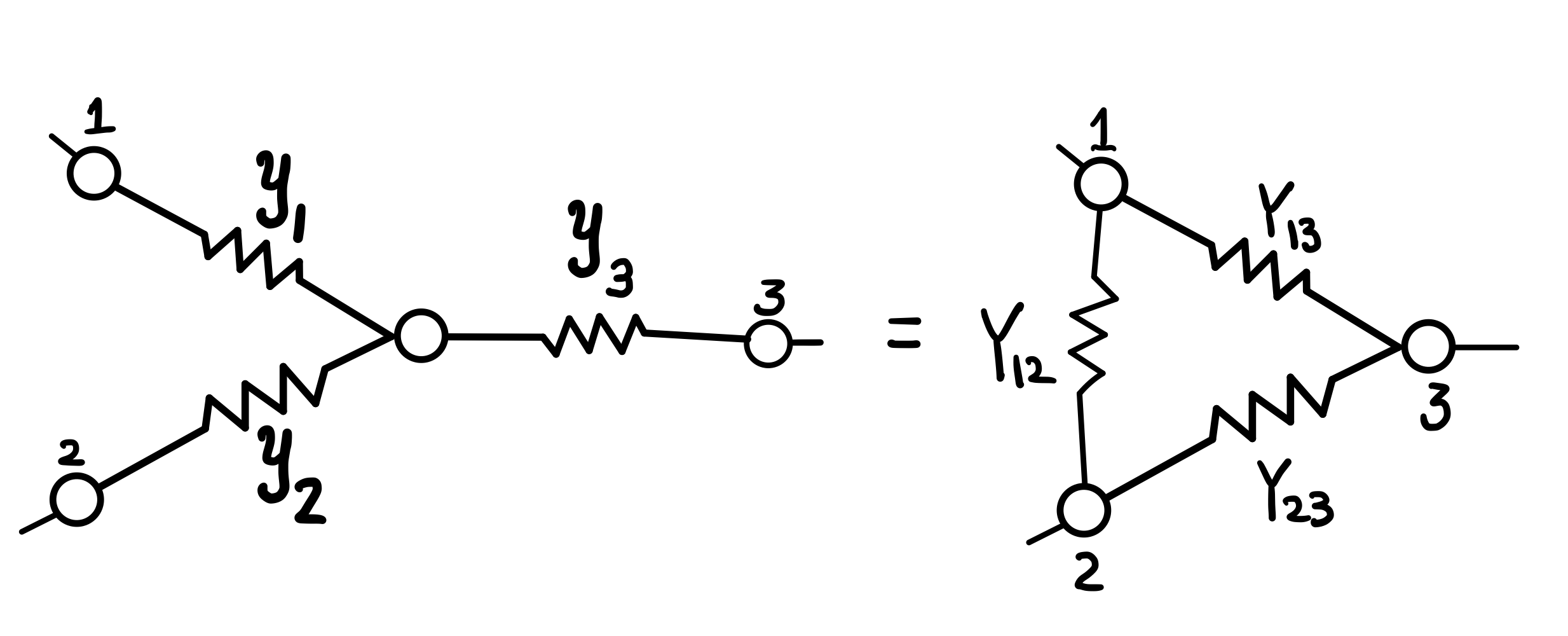}
        {\caption*{(Y/$\Delta$)$_3$}}
    \end{minipage} 
\end{figure}
\FloatBarrier

\section{Rewriting for ${\sf Resist}_R$}

Given any resistor circuit with non-zero conductance on each wire, one can reduce the circuit to a family of meshes, thus to a normal form, by removing all the parallel resistors and the internal nodes using the identities of ${\sf Resist}_R$. Our goal in this section is to prove that the resulting rewriting systems terminates and is confluent, which is the main result of this paper. In order to prove the main result, we observe the following. 

\begin{lemma}  
\label{Lemma: unit-star rewrite}
The unit and star-mesh critical pairs in ${\sf Resist}_R$ resolve. 
\end{lemma}
\begin{proof}
Consider the star network composed on one of its outgoing edges with the unit, see circuit  (a). 

\[ ( a)~~~ \vcenteredinclude{scale=0.05}{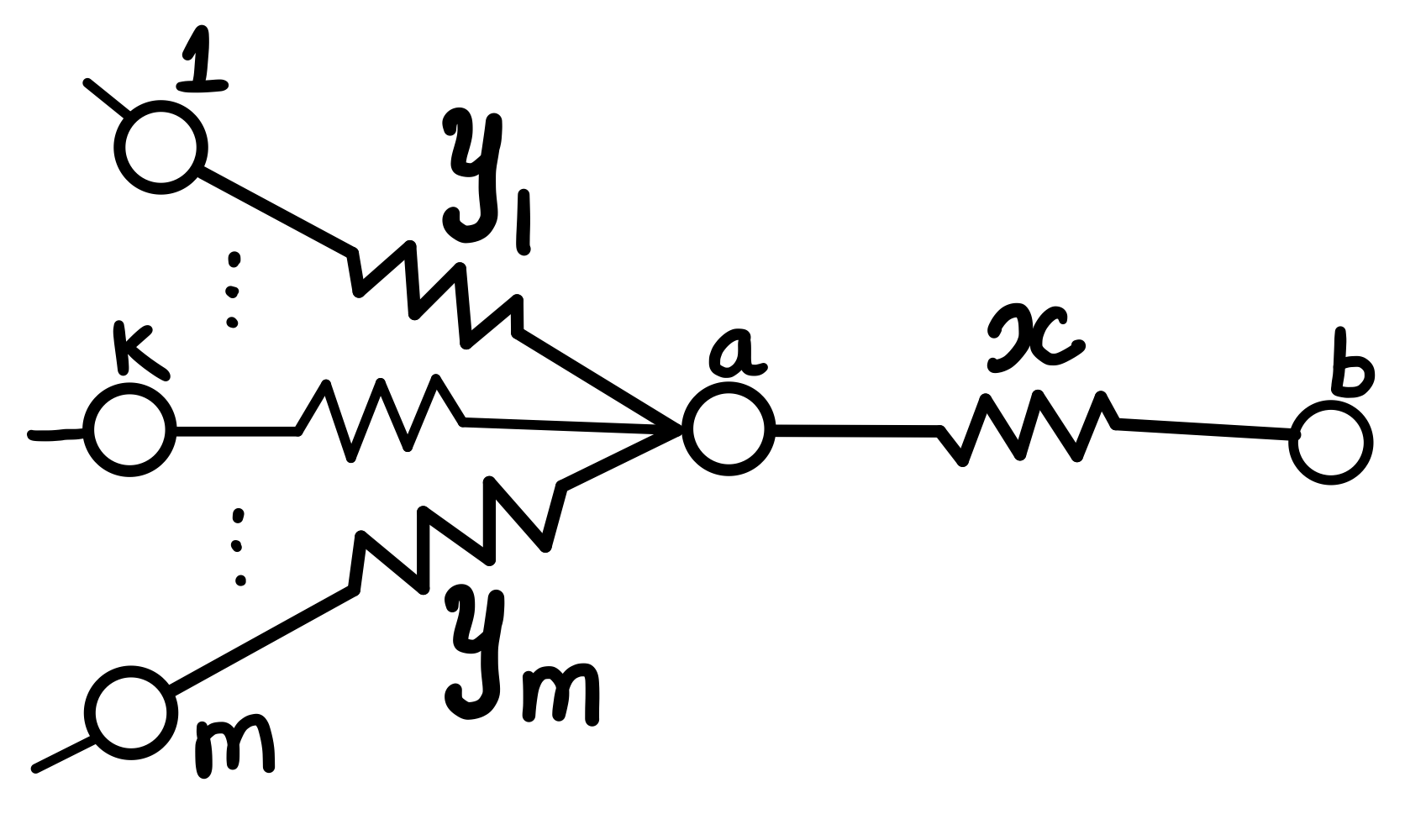} \quad \quad
    (b)~~~ \vcenteredinclude{scale=0.05}{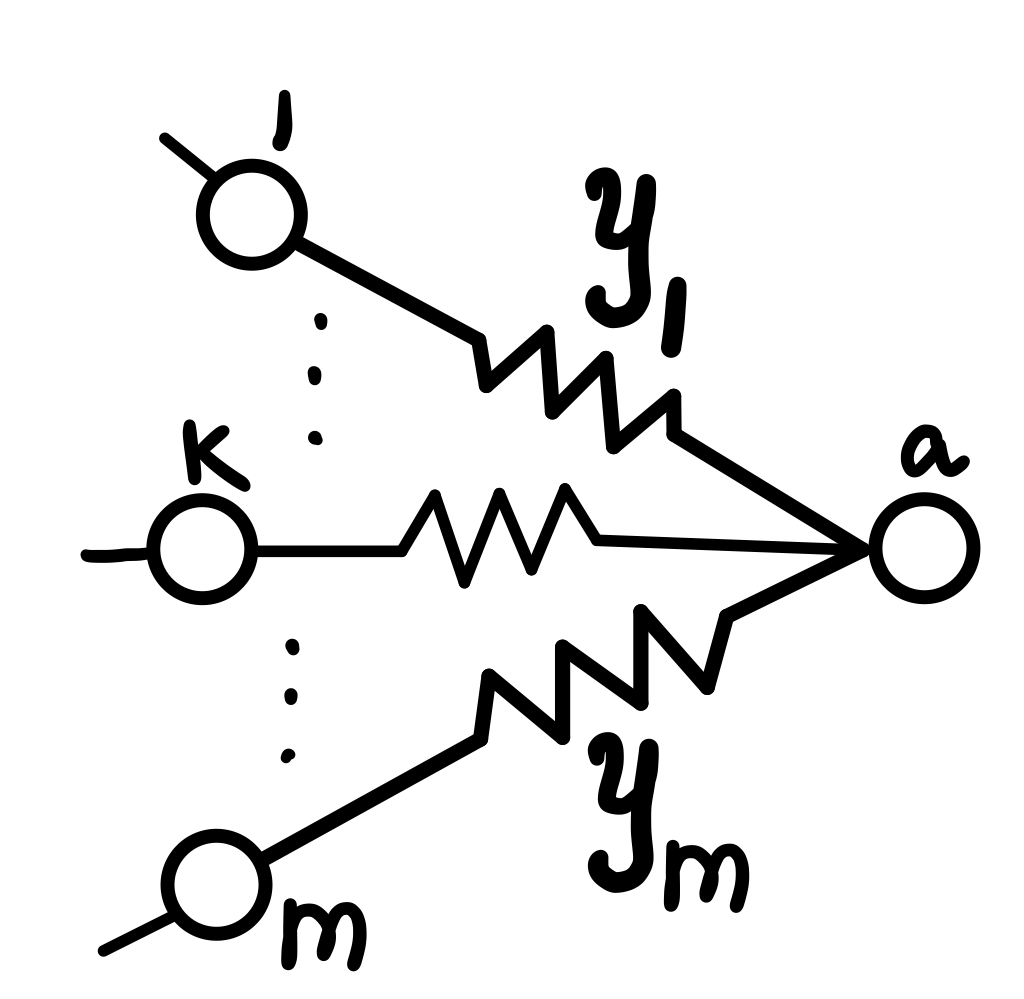} 
\]

To rewrite circuit $(a)$, node $b$ shall be eliminated first by applying $(\text{Y}/\Delta)_1$ or node $a$ shall be eliminated first by applying $(\text{Y}/\Delta)_{m+1}$, thereby resulting in a critical pair. 

Resolving the node $b$ first by applying $(\text{Y}/\Delta)_1$ results in an $m$-node star network (with one internal node $a$), see circuit (b) above. Applying $(Y/\Delta)_m$ to circuit (b) to eliminate node $a$ results in a mesh in which for each pair of nodes $1 \leq i,j \leq m$, the resistor edge connecting them has conductance $Y_{ij}^{ba}$ with value: 
\[Y_{ij}^{ba}  = \frac{y_i y_j}{\sum_k^m y_k} \]

On the other hand, resolving node $a$ first by applying $(Y/\Delta)_{m+1}$ results in a mesh network with $m$ external nodes, each one of which are connected to the internal node $b$, see the figure below. In the resulting circuit, each external node $i$ is connected to node $b$ via a resistor with conductance $Y_{ix}^a$, see equation \ref{eqn: R1}-(a). Every pair of external nodes $i$ and $j$ are connected by a resistor with conductance $Y_{ij}^{a}$, see \ref{eqn: R1}-(b). 
\[ \includegraphics[scale=0.055]{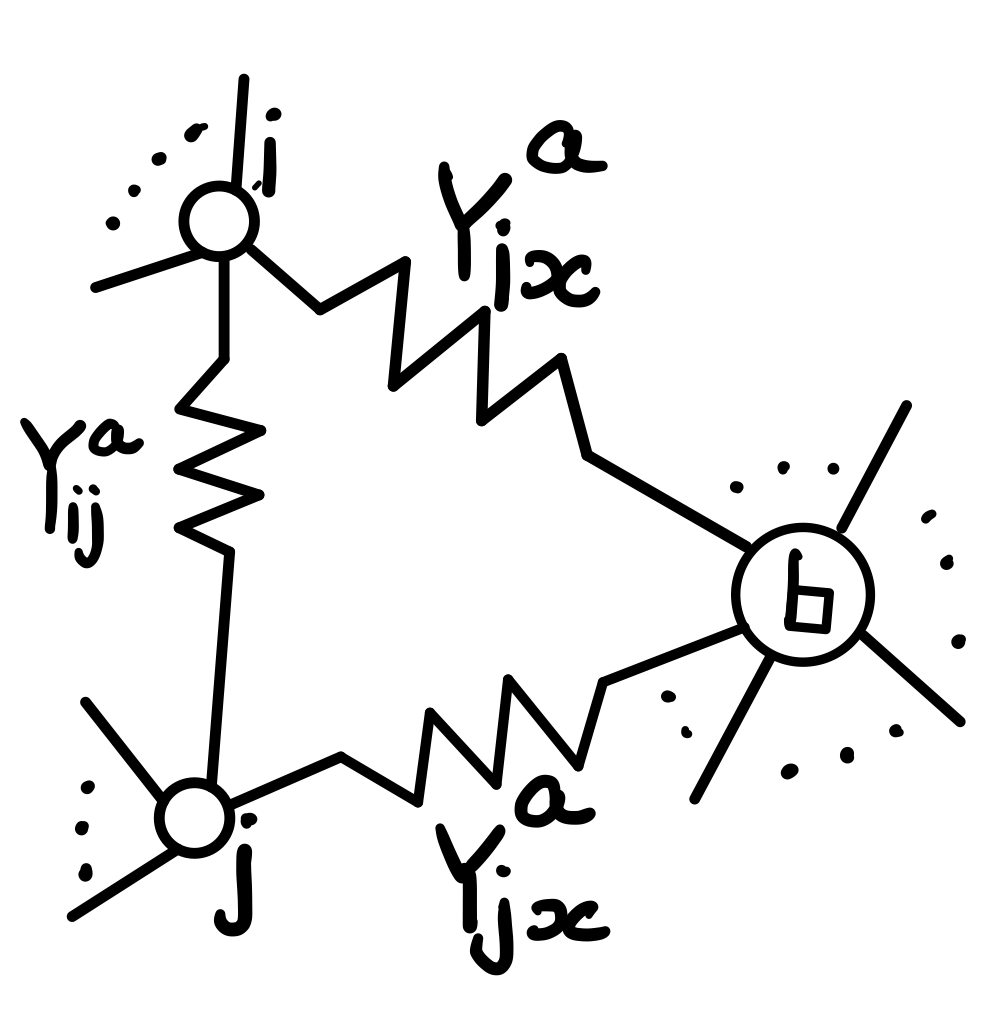} \]
\begin{align} 
\label{eqn: R1}
(a)~~~ Y_{ij}^a = \frac{y_i y_j}{\sum_k^m y_k + x} ~~~~~~~~~ 
(b) ~~~ Y_{ix}^a = \frac{y_i x}{\sum_k^m y_k + x} 
\end{align}

Now, applying $(Y/\Delta)_m$ to resolve node $b$ in the resulting circuit, leads to parallel conductances $Y_{ij}^a$ and $Y_{ij}^{ab}$ between any two nodes $1 \leq i, j \leq m$, see the diagram below. 
\[\vcenteredinclude{scale=0.05}{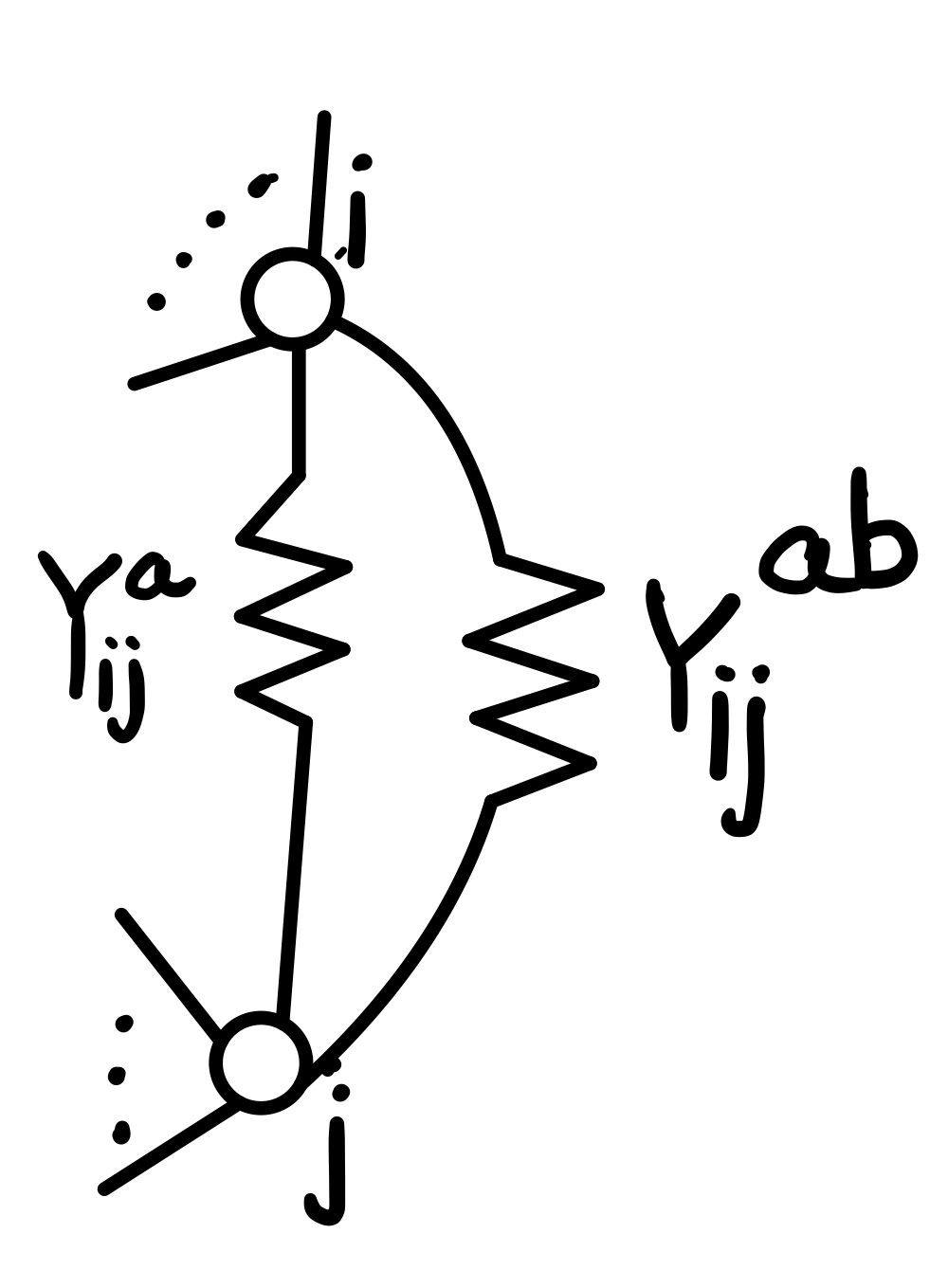}\] 
The value of $Y_{ij}^{ab}$ is computed as follows:
\[ Y_{ij}^{ab} = \frac{Y_{ix}^a Y_{jx}^a}{\sum_k^m Y_{kx}^a} 
= \frac{\frac{y_i x y_j x}{(\sum_k^m y_k + x)^2}}{\frac{\sum_k^m y_k x}{(\sum_k^m y_k + x)}} = \frac{\frac{y_i y_j x}{(\sum_k^m y_k + x)}}{(\sum_k^m y_k )} 
= \frac{y_i y_j x}{(\sum_k^m y_k )(\sum_k^m y_k + x)}\]

Combining the parallel edges, 
\begin{align*}
Y_{ij}^a + Y_{ij}^{ab} &= \frac{y_i y_j}{\sum_k^m y_k + x} + \frac{y_i y_j x}{(\sum_k^m y_k)(\sum_k^m y_k + x)} = \frac{y_i y_j (\sum_k^m y_k) + y_i y_j x}{(\sum_k^m y_k)(\sum_k^m y_k + x)} \\ 
&= \frac{y_i y_j(\sum_k^m y_k+x)}{(\sum_k^m y_k)(\sum_k^m y_k + x)}  = \frac{y_i y_j}{(\sum_k^m y_k)} = Y_{ij}^{ba}
\end{align*}

Thus the two orders of rewriting the circuit in diagram (a) produce equivalent results.
\end{proof}

\begin{lemma}
    \label{Lemma: star-parallel rewrite}
    The parallel and star-mesh critical pairs in ${\sf Resist}_R$ resolve.
 \end{lemma}
\begin{proof}
Consider the star network composed with two of its edges connected in parallel, see circuit (a). 

\[ ( a)~~~ \vcenteredinclude{scale=0.05}{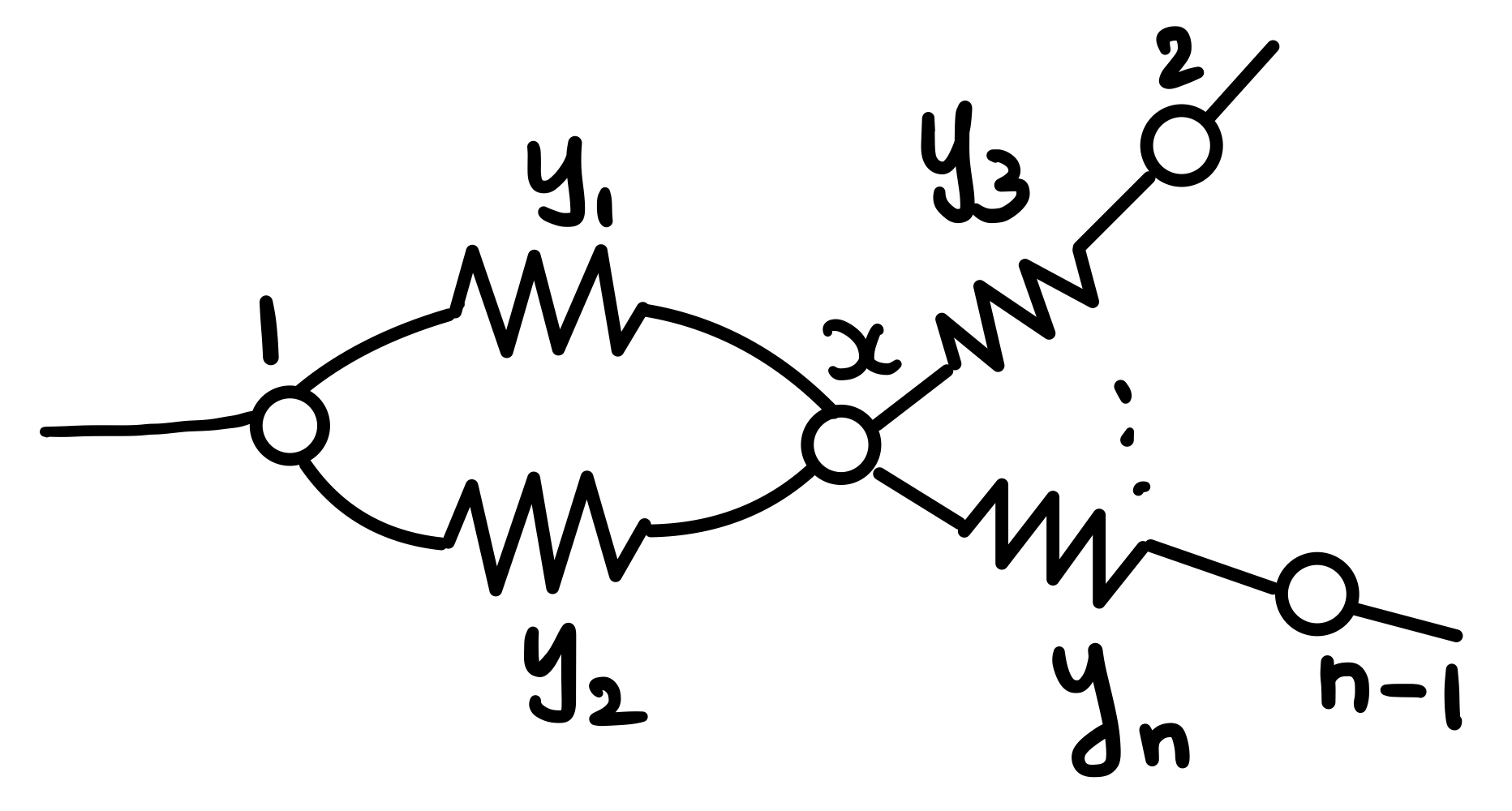} \quad \quad
    (b)~~~ \vcenteredinclude{scale=0.05}{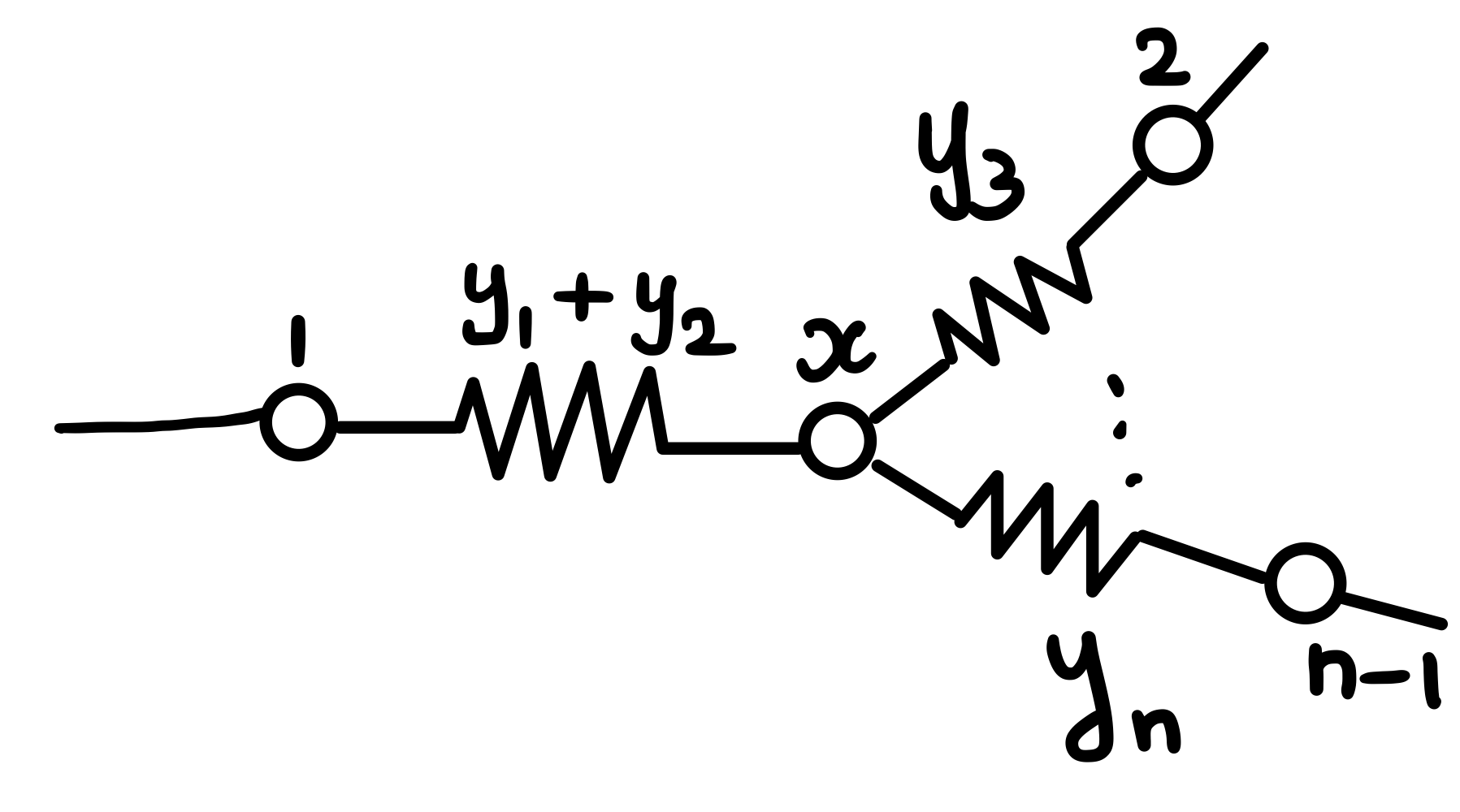} 
\]

To rewrite circuit $(a)$, one can first apply {\sf [Parallel]} rewrite rule resulting in circuit $(b)$. Applying (Y/$\Delta$)$_{n-1}$ on circuit $(b)$ results in a mesh with $n-1$ nodes with conductance value between any two nodes as follows: For all $2 \leq j, k \leq n-1$,

\begin{align*}
    (a)~~ Y_{1k} = \frac{(y_1 + y_2) y_k}{\sum_i^{n} y_i} ~~~~~~~~ (b)~~~ Y_{jk} = \frac{y_j y_k}{\sum_i^{n} y_i}
\end{align*}

Alternatively, to rewrite circuit $(a)$, one can apply the spider {\sf [Expansion]} rule resulting in circuit $(c)$ shown below. Applying, (Y/$\Delta$)$_{n}$ on circuit $(c)$, results in  $n$-node mesh with a parallel edge of infinite conductance between nodes $1$ and $s$ as shown in circuit $(d)$. Applying {\sf[Short circuit]} rule on circuit $(d)$ results in parallel edge between node $1$, and any other node $k$ where $2 \leq k \leq$, see circuit $(f)$. In the pictures below we use $\sigma$ to denote the sum of the conductances, $\sigma=\sum_{i=1}^{n}y_{i}$

\[ (c)~~~ \vcenteredinclude{scale=0.07}{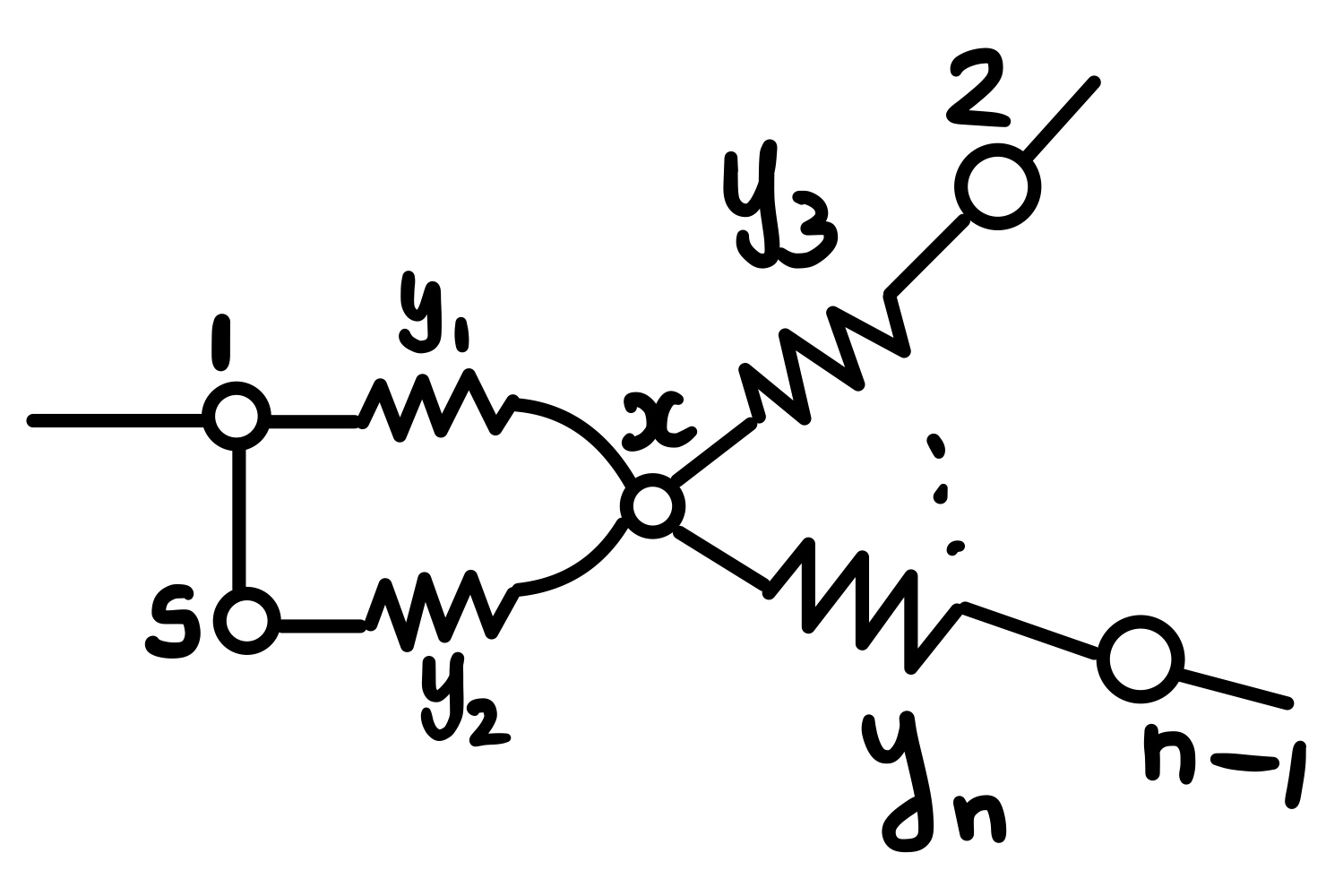} \quad \quad
   (d)~~~ \vcenteredinclude{scale=0.07}{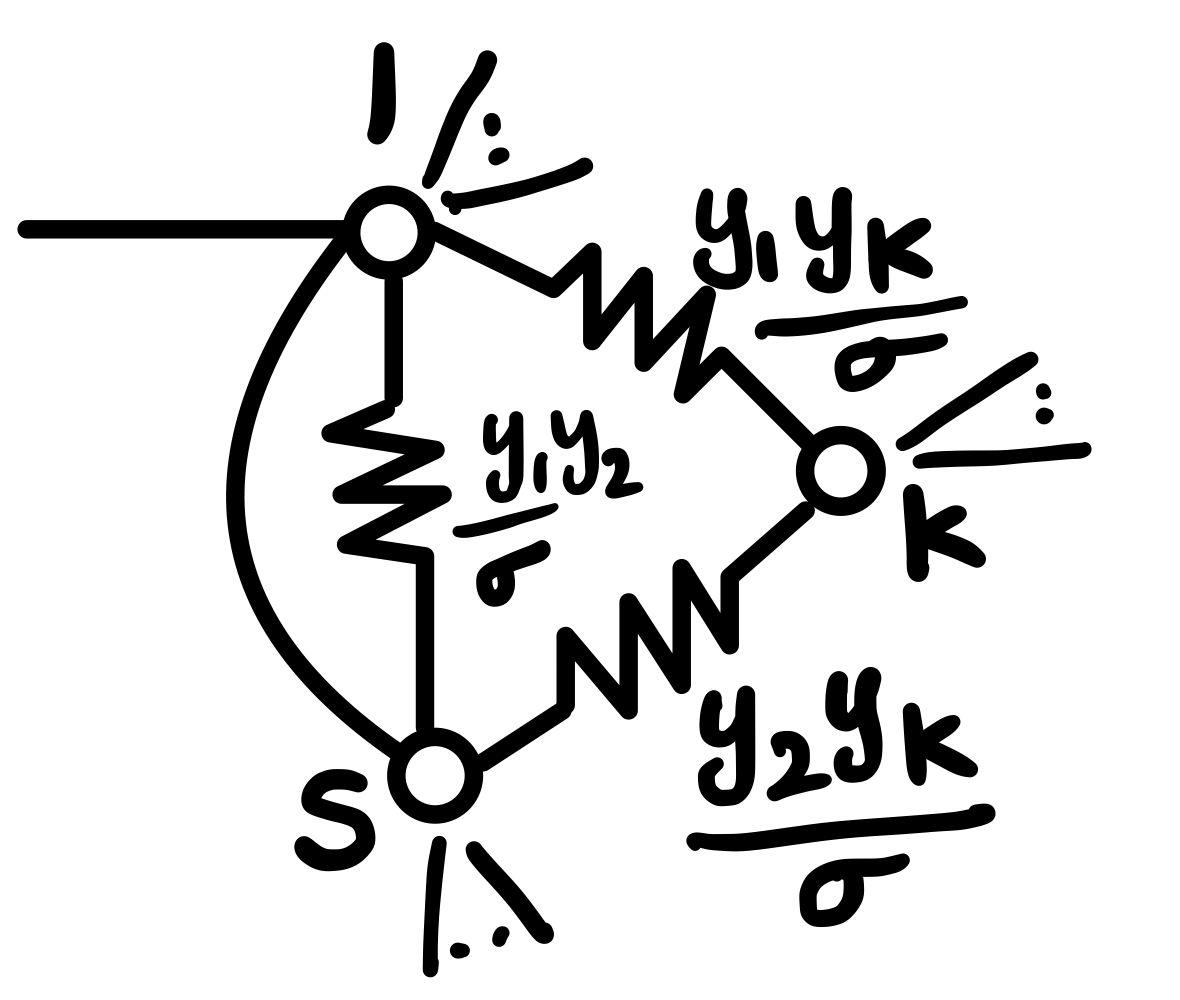} \quad \quad
   (f)~~~ \vcenteredinclude{scale=0.07}{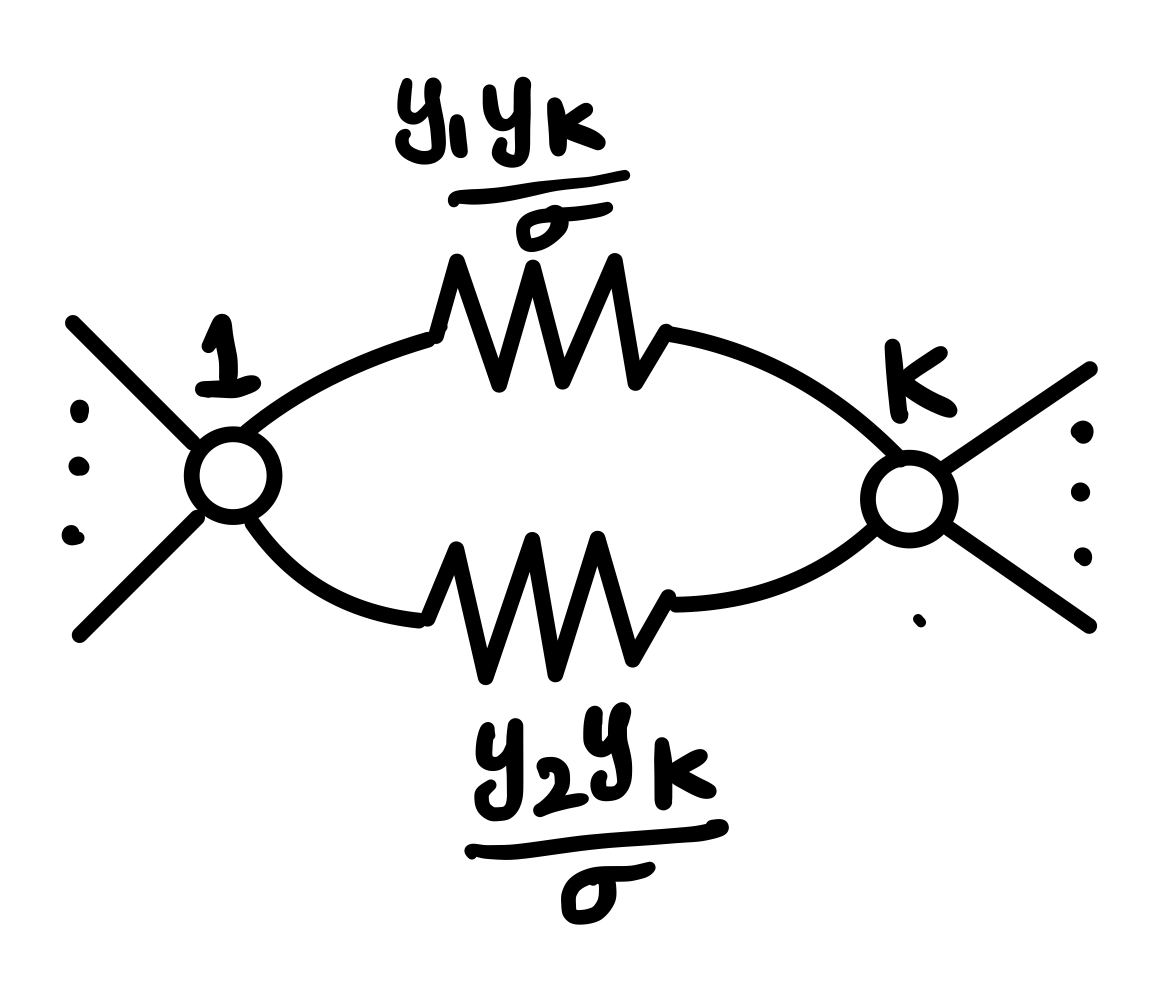} 
\]
Using {\sf [Parallel]} rule in the mesh, for all $2 \leq k \leq n-1$, 
\[ Y_{1k} = \frac{(y_1 + y_2) y_k}{\sum_i^n y_i} \]
Moreover, by (Y/$\Delta$)$_{n}$ rule, for all $2 \leq j,k \leq n-1$,
\[ Y_{jk} = \frac{y_j y_k}{\sum_i^{n} y_i} \]
\end{proof}

\begin{lemma}
\label{Lemma: star-star rewrite}
The star-mesh critical pairs in ${\sf Resist}_R$ can be resolved.
\end{lemma}

\begin{proof} 

Consider the circuit shown below. We must show that eliminating node $a$ first by applying $(Y/\Delta)_{m+1}$ and then node $b$ by applying $(Y/\Delta)_{n+m}$ yields the same result as eliminating node $b$ first by applying $(Y/\Delta)_{n+1}$ and then node $a$ by applying $(Y/\Delta)_{n+m}$. 

\[ \includegraphics[scale=0.05]{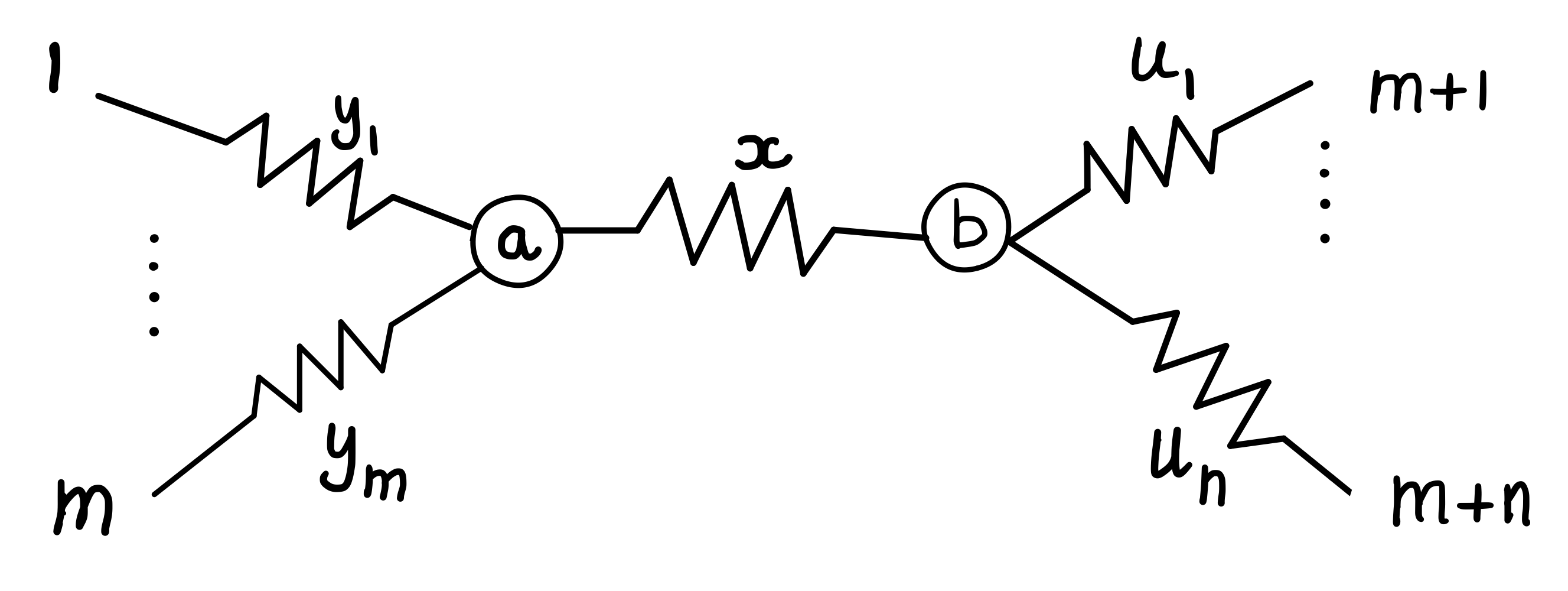}\]

Note that, in order to rewrite a star network all its arms must have non-zero conductances (these would result in an open circuit) nor can there be any infinite conductances (these would be a bare wire).

We first develop a general algorithm for naming the new edges resulting from each rewrite of the circuit. The algorithm is as follows:

{\bf Algorithm for labelling edges:} 
\begin{enumerate}
\item At each step, the new edges always carry the names of the eliminated nodes as superscript in the order of their elimination. 
\item The subscripts (in general) refer to the index of resistors being combined -- $Y_{ij}^a$ is given by combining $y_i$ and $y_j$ by eliminating $a$ with $i \leq j$; combining the edges $Y_{ix}^a$ and $Y_{jx}^a$ by eliminating $b$ gives $Y_{ij}^{ab}$; $Y_{ix}^a$ is given by combining the edge $y_i$ with $x$ eliminating node $a$. 
\item Whenever only `$y_i$' resistors are combined create to a new edge, the new edge carries label `Y' with appropriate super and subscripts; label `$\phi$' means a $y$ and a $u$ resistor are combined; label `$U$' means only $u$ resistors have been combined. 
\end{enumerate}

To make the naming procedure clear, we consider a simple case shown in Figure \ref{fig:confluence example}:

\begin{figure}[h]
    \centering
    \includegraphics[scale=0.05]{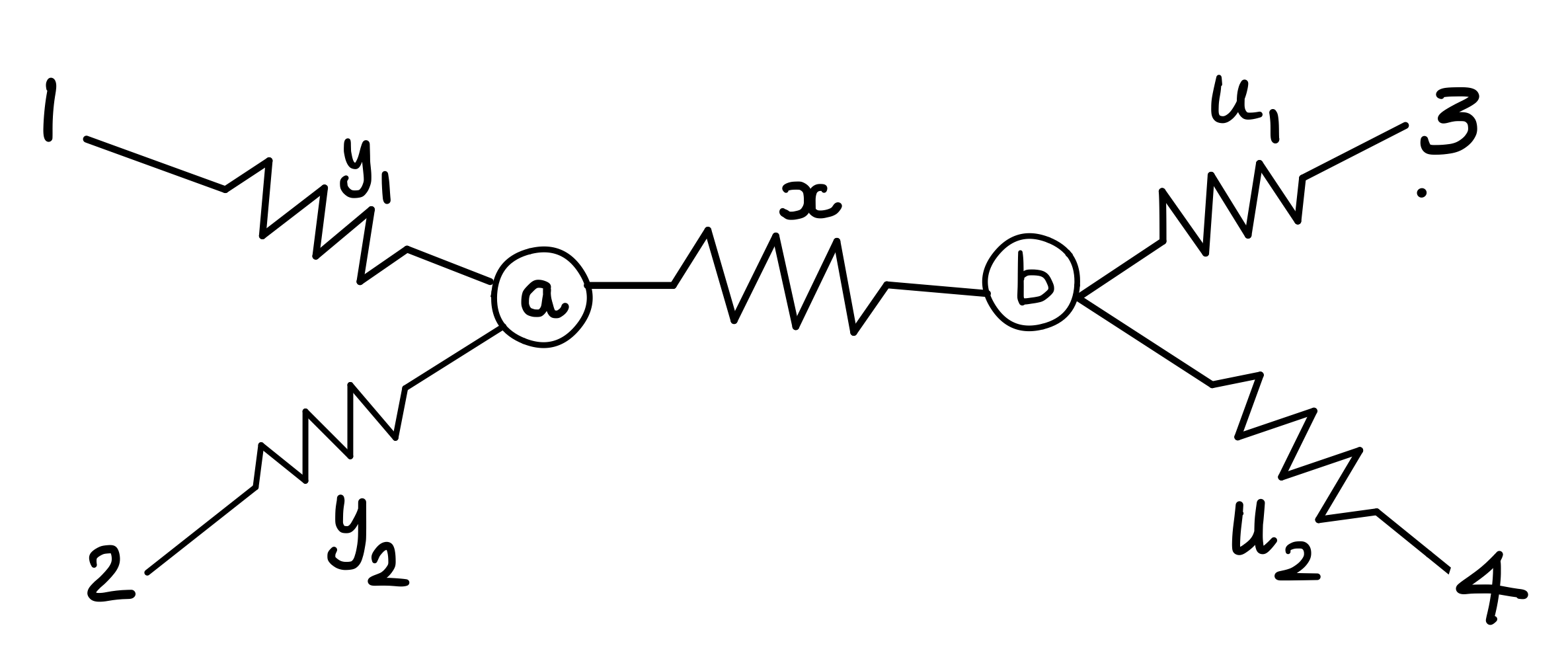}
    \caption{Node $a$ and node $b$ connected by resistor $x$}
    \label{fig:confluence example}
\end{figure} 
\FloatBarrier 

Two possible ways of reducing the circuit in Figure \ref{fig:confluence example} are shown below: 

Eliminate internal node $a$ followed by internal node $b$:
\[ \vcenteredinclude{scale=0.04}{diag/Confluence_3.png} \xRightarrow{(Y/\Delta)_3}
   \vcenteredinclude{scale=0.08}{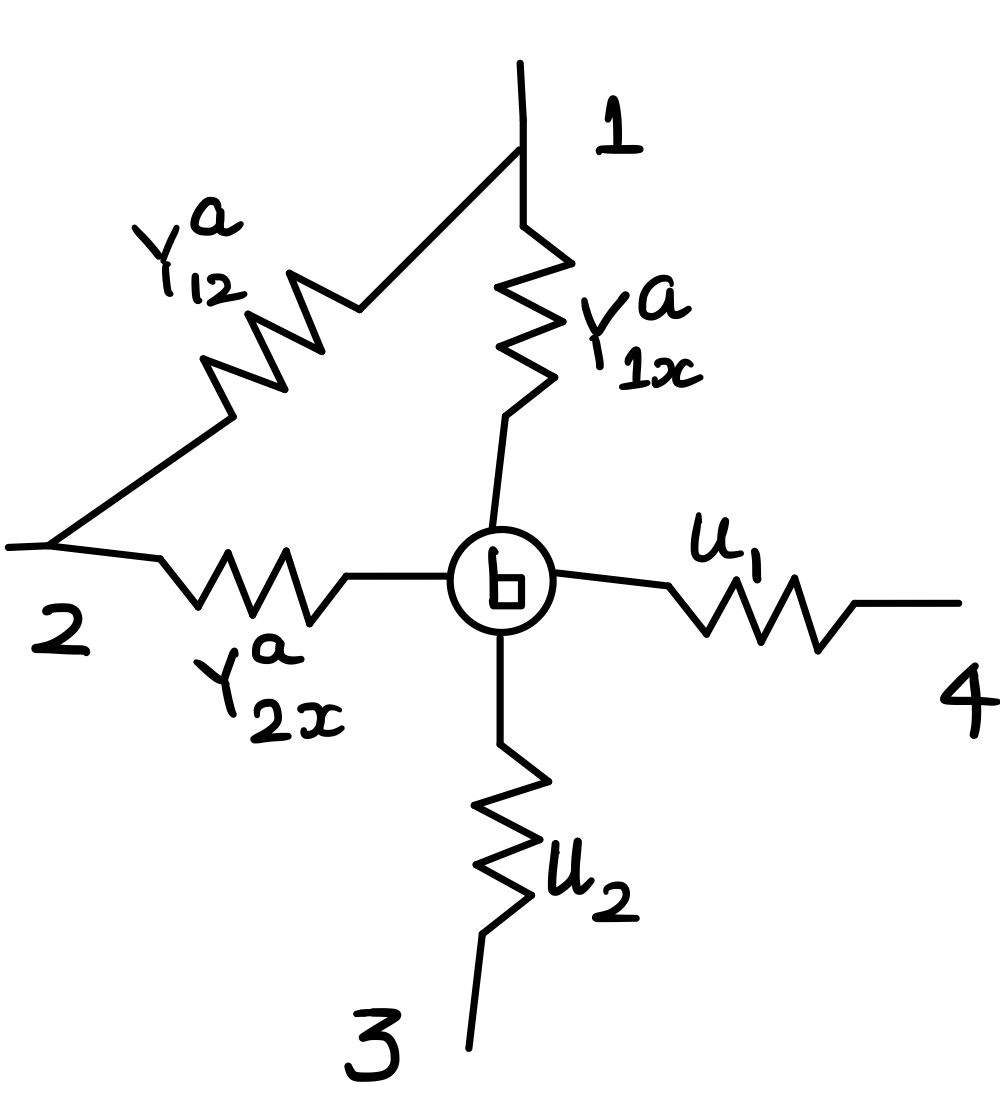} \xRightarrow{(Y/\Delta)_4}
   \vcenteredinclude{scale=0.082}{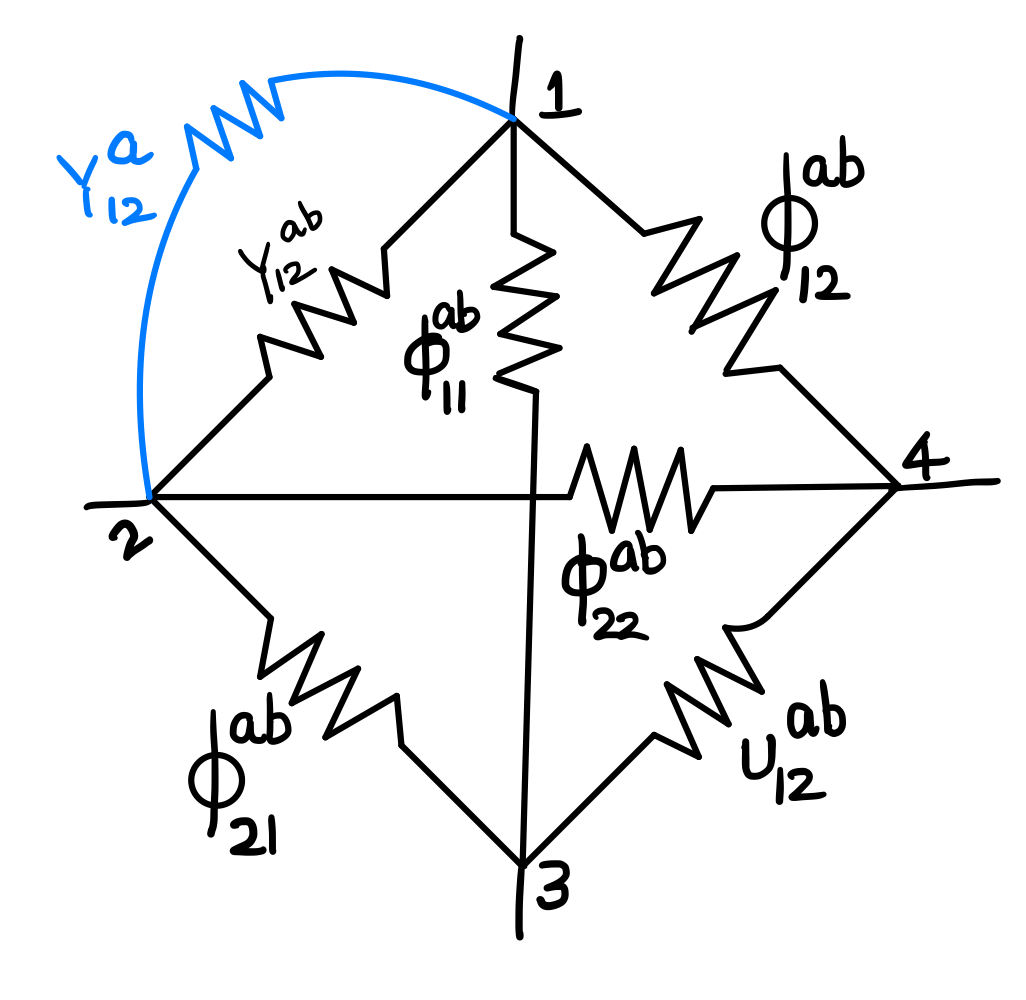} \xRightarrow{\sf[Parallel]}
   \vcenteredinclude{scale=0.082}{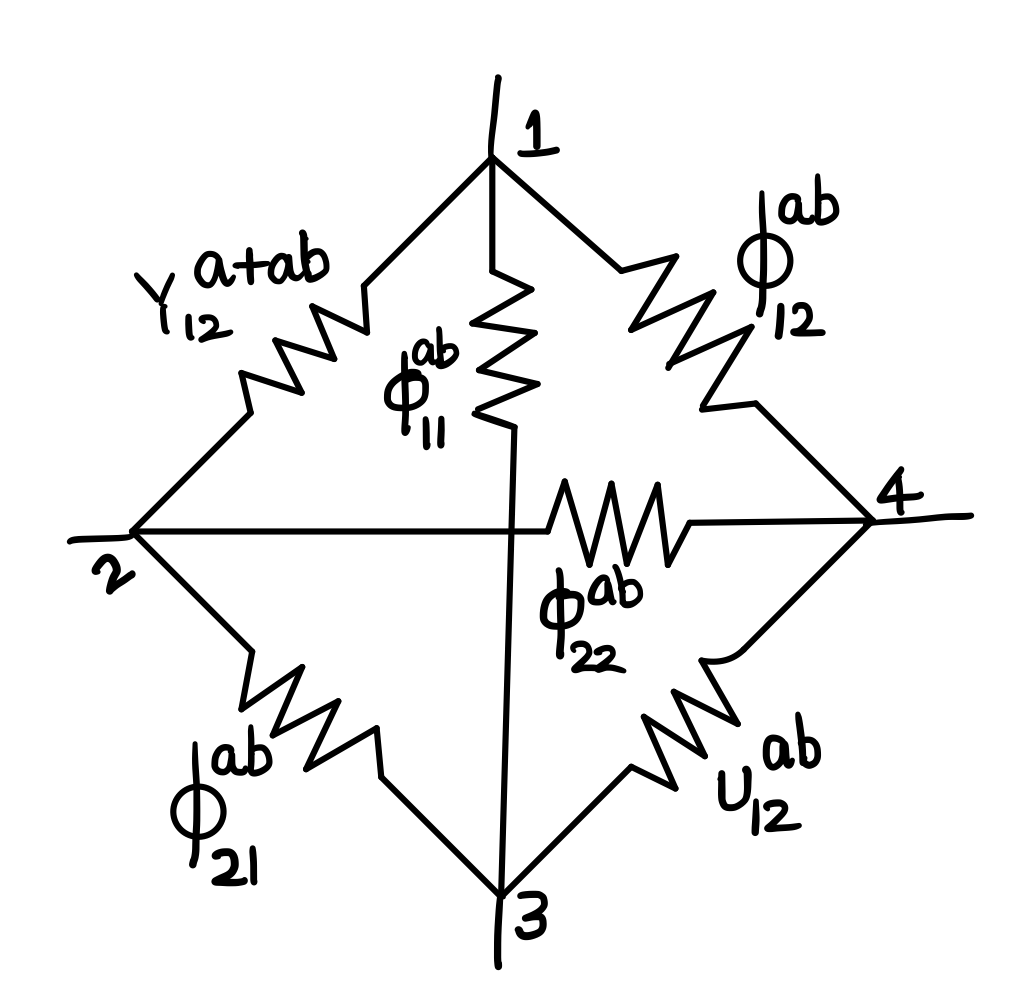} \]

Eliminate internal node $b$ followed by internal node $a$:
\[ \vcenteredinclude{scale=0.04}{diag/Confluence_3.png} \xRightarrow{(Y/\Delta)_3} 
   \vcenteredinclude{scale=0.08}{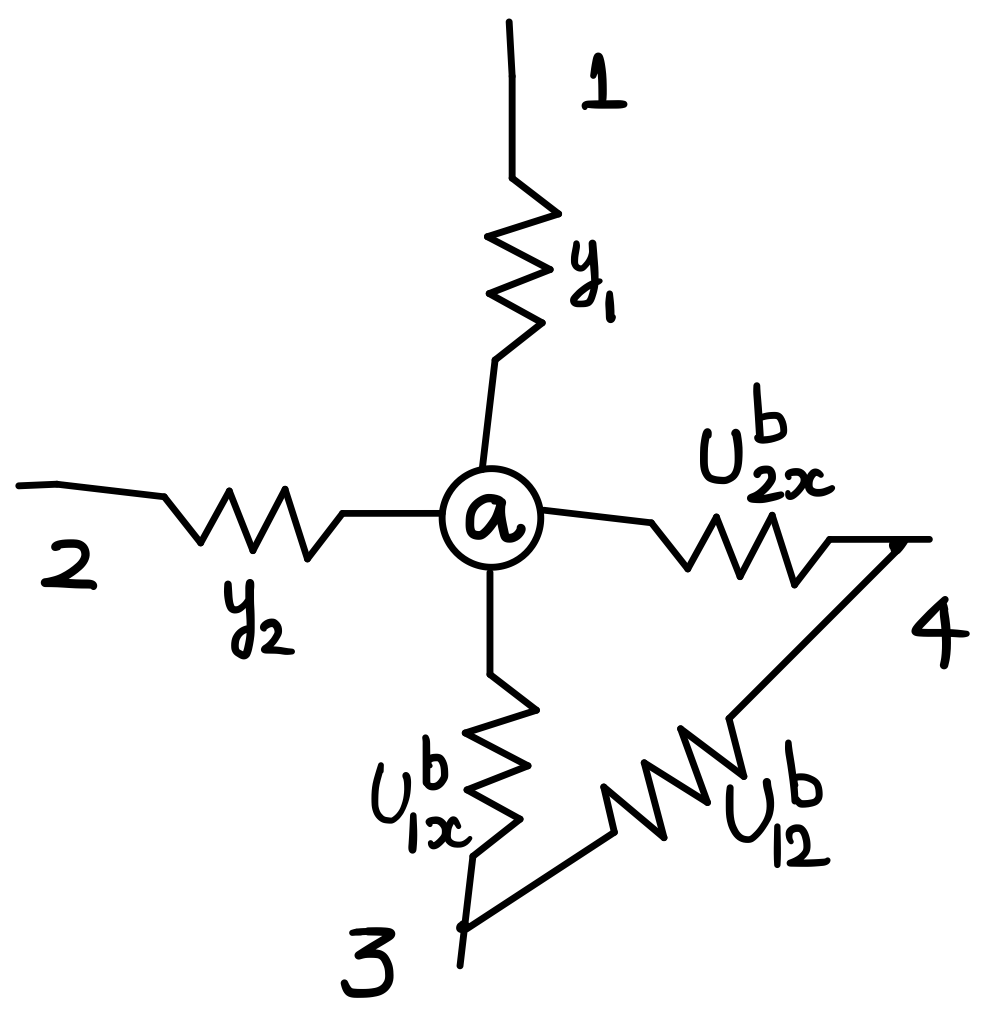} \xRightarrow{(Y/\Delta)_4}
   \vcenteredinclude{scale=0.082}{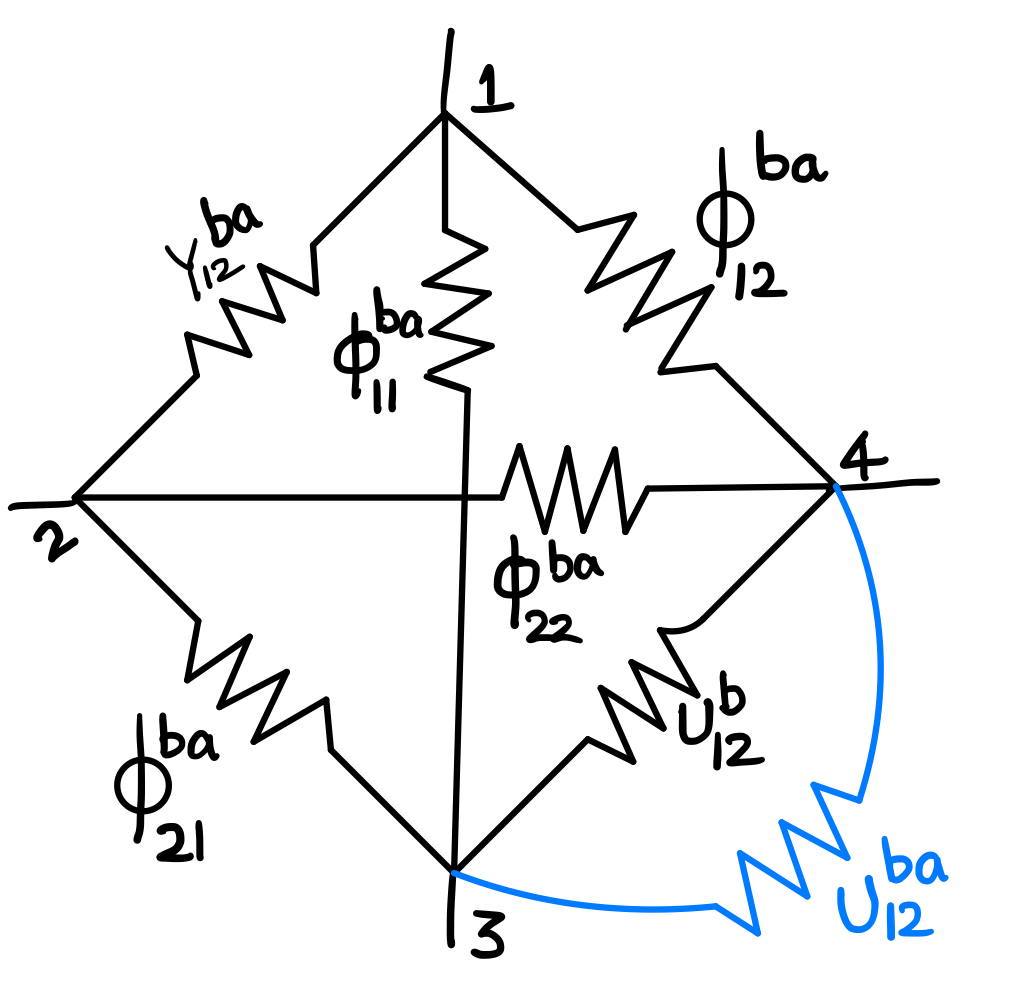} \xRightarrow{\sf[Parallel]}
   \vcenteredinclude{scale=0.082}{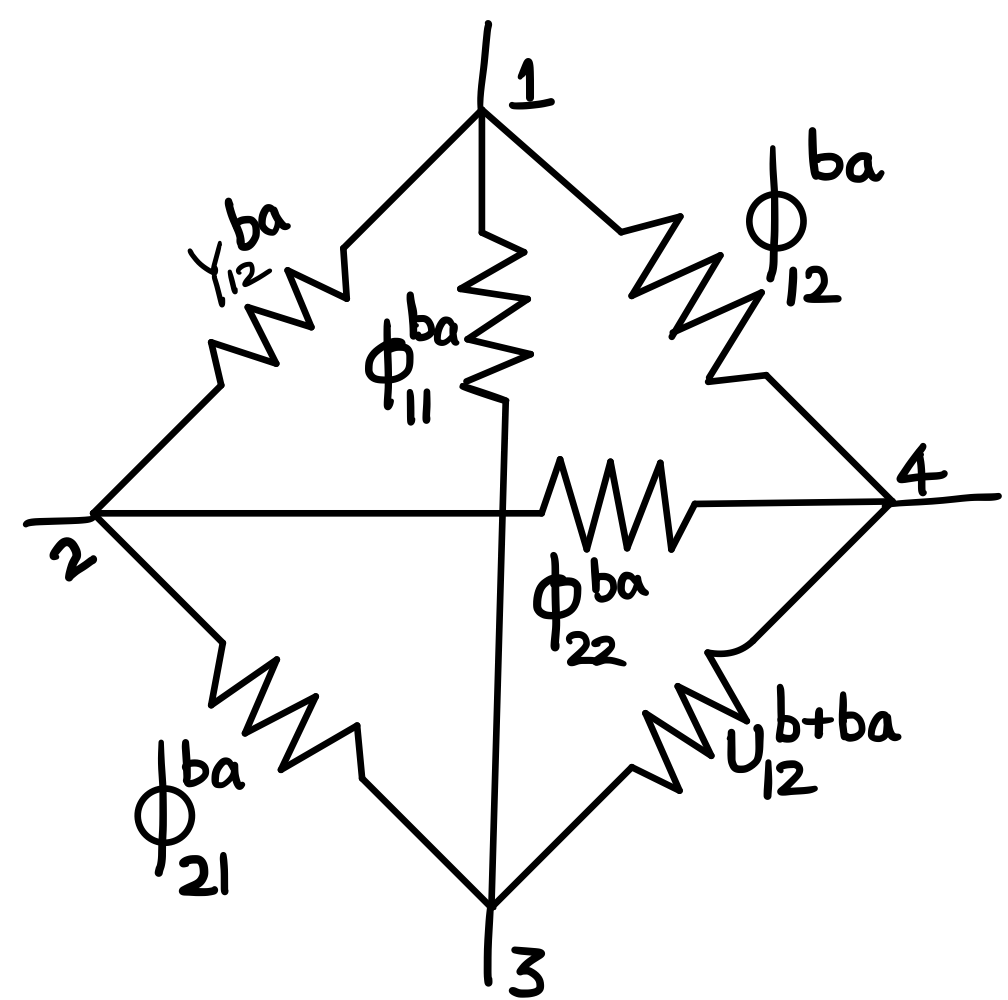} \]

The two ways of reducing the circuit are equal if:
    \begin{align*}
    \text{for } 1 \leq i,j \leq 2, ~~~  \phi^{ab}_{ij} &= \phi^{ba}_{ij}  \\
    \text{for } 1 \leq i,j \leq 2,~~~     U^{ab}_{ij} &= U^{b+ba}_{ij} \\
    \text{for } 1 \leq i,j \leq 2,~~~    Y^{a+ab}_{ij} &= Y^{ba}_{ij}  
    \end{align*}

Following this approach, the two ways of reducing the general circuit (given in the beginning of this proof) are equal if:
   \begin{align}
    \label{eqn: confluence-1}
    \text{for } 1 \leq i \leq m, ~~~ 1 \leq j \leq n ~~~   \phi^{ab}_{ij} &= \phi^{ba}_{ij}  \\
    \label{eqn: confluence-3}
    \text{for } 1 \leq i,j \leq n ~~~     U^{ab}_{ij} &= U^{b+ba}_{ij} \\
    \label{eqn: confluence-2}
    \text{for } 1 \leq i,j \leq m ~~~    Y^{a+ab}_{ij} &= Y^{ba}_{ij}  
    \end{align}
where,

\begin{minipage}[b]{0.45\textwidth} 
    \begin{align*}
    \text{for } 1 \leq i,j \leq m, ~~~ Y^{a+ab}_{ij} &= Y^{a}_{ij}+Y^{ab}_{ij}
    \end{align*}
\end{minipage} \hfill
\begin{minipage}[b]{0.45\textwidth}
    \begin{align*}
     \text{for } 1 \leq i,j \leq n, ~~~ U^{ba+a}_{ij} &= U^{b}_{ij}+U^{ba}_{ij} 
    \end{align*}
\end{minipage}
\begin{align*} 
\text{for } 1 \leq i \leq m ~\text{ and }~ 1 \leq j \leq n, ~~~~~  
\phi^{ab}_{ij} &= \frac{Y^{a}_{ix}u_{j}}{\sum_{k}^{n}u_{k}+\sum_{k}^m Y^{a}_{kx}} \\
\text{for } 1 \leq i \leq m ~\text{ and }~ 1 \leq j \leq n, ~~~~~
{\phi}_{ij}^{ba} &= \frac{y^{b}_{i} U^{b}_{jx}}{\sum_{k}^m y_{k}+\sum_{k}^{n}U^{b}_{kx}} 
\end{align*}
\begin{minipage}[b]{0.45\textwidth} 
    \begin{align*}
     \text{for } 1 \leq i,j \leq m, ~~~ Y^{ab}_{ij} 
     &= \frac{Y^{a}_{ix}Y^{a}_{jx}}{\sum_{k}^{n}u_{k}+\sum_{k}^m Y^{a}_{kx}} \\ 
     \text{for } 1 \leq i,j \leq n, ~~~ U^{ab}_{ij} 
     &= \frac{u_{i}u_{j}}{\sum_{k}^{n} u_{k}+\sum_{k}^m Y^{a}_{kx}}  \\ 
         \text{for } 1 \leq i,j \leq m, ~~~ Y^{a}_{ij} &= \frac{y_{i}y_{j}}{\sum_{k}^m y_{k}+x} \\
     \text{for } 1 \leq i \leq m, ~~~ Y^{a}_{ix} &= \frac{y_{i}x}{\sum_{k}^m y_{k}+x} 
    \end{align*}
\end{minipage} \hfill
\begin{minipage}[b]{0.45\textwidth}
    \begin{align*}
    \text{for } 1 \leq i,j \leq n, ~~~ U_{ij}^{ba} 
    &= \frac{U_{ix}^b U_{jx}^b}{\sum_{k}^m y_{k}+\sum_{k}^{n}U^{b}_{kx}} \\ 
    \text{for } 1 \leq i,j \leq m, ~~~ Y^{ba}_{ij} 
    &= \frac{y^{a}_{i} y^{a}_{j}}{\sum_{k}^m y_{k}+\sum_{k}^{n}U^{b}_{kx}} \\
    \text{for } 1 \leq i,j \leq n, ~~~ U^{b}_{ij} &= \frac{u_{i}u_{j}}{\sum_{k}^{n}u_{k}+x} \\
    \text{for } 1 \leq i \leq n, ~~~  U^{b}_{ix} &= \frac{u_{i}x}{\sum_{k}^{n}u_{k}+x} 
  \end{align*}
\end{minipage} 
\FloatBarrier

\vspace{1em}

Now, proving \ref{eqn: confluence-1}, $\phi_{ij}^{ab} = \phi_{ij}^{ba}$:
\begin{align*}
\phi^{ab}_{ij} 
&= \frac{ \frac{y_{i}x}{\sum_{k}^m y_{k}+x} u_{j}}{\sum_{k}^{n}u_{k}+\sum_{k}^m \frac{y_{k}x}{y_{k}+x}} 
=  \frac{ y_{i} x u_{j}}{(\sum_{k}^{n}u_{k})(\sum_{k}^m y_{k}+x) +\sum_{k}^m y_{k}x} \\
&=  \frac{ y_{i} x u_{j}}{(\sum_{k}^m y_{k})(\sum_{k}^{n}u_{k}) + \sum_{k}^{n}u_{k}x 
+ \sum_{k}^m y_{k}x } 
=  \frac{ y_{i} x u_{j}}{ \sum_{k}^{n} u_{k} x
+ \sum_{k}^m y_{k}( \sum_{k}^{n}u_{k} + x) } \\
&= \frac{ \frac{y_{i} x u_{j}}{\sum_{k}^{n}u_{k} + x} }{\frac{\sum_{k}^{n} u_{k} x}{\sum_{k}^{n}u_{k} + x} 
+ \sum_{k}^m y_{k} } 
= \frac{ y_{i} U_{jx}^b }{\sum_{k}^{n} U_{kx}^b
+ \sum_{k}^m y_{k} } = \phi_{ij}^{ba}
\end{align*}

Now, proving \ref{eqn: confluence-3}, $U_{ij}^{ab} = U_{ij}^{b+ba}$:

\begin{align}
\label{eqn: LHS}
U^{ab}_{ij} &= \frac{u_{i}{u_j}}{\sum_{k}^{n} u_k + \sum_k^m Y_{kx}^a} = 
\frac{u_{i}u_j}{\sum_{k}^{n} u_k + \frac{\sum_k^m Y_{k}x}{\sum_k^m y_k +x}} = 
\frac{u_{i}u_{j}}{\sum_{k}^{n} u_k ({\sum_k^m y_k +x}) + \sum_k^m Y_{k}x}  \notag \\ 
& \stackrel{(*)}{=} \frac{u_{i}u_{j}}{\sum_k^m y_k ({\sum_{k}^{n} u_k +x}) + \sum_{k}^{n} u_{k}x} 
= \frac{\frac{u_i u_j}{\sum_{k}^{n} u_k + x}(\sum_k^m y_k+x)}
{\sum_k^m y_k  + \frac{\sum_{k}^{n} u_k x}{ \sum_{k}^{n} u_k +x}} 
\end{align}

For step $(*)$, see computation of denominator in the proof of \ref{eqn: confluence-1}, $\phi_{ij}^{ab} = \phi_{ij}^{ba}$.

\begin{align}
\label{eqn: RHS}
    U^{b + ba}_{ij} =  U_{ij}^b + U_{ij}^{ba} 
    &= \frac{u_i u_j}{\sum_{k}^{n} u_k + x} + \frac{\frac{u_i u_j x^2}{(\sum_k^n u_k + x)^2}}{\sum_k^m y_k + \frac{\sum_{k}^{n} u_k x}{\sum_{k}^{n} u_k + x}} 
    = \frac{u_i u_j\left( \sum_k^m y_k + \frac{\sum_{k}^{n} u_k x}{\sum_{k}^{n} u_k + x}  \right) + \left( \frac{u_i u_j x^2}{\sum_k^n u_k + x}  \right)}{\left(\sum_{k}^{n} u_k + x \right)\left( \sum_k^m y_k + \frac{\sum_{k}^{n} u_k x}{ \sum_k^m y_k + x} \right)} \notag \\ 
    &= \frac{u_i u_j\left( \sum_k^m y_k + \frac{\sum_{k}^{n} u_k x}{\sum_{k}^{n} u_k + x}   
    +  \frac{x^2}{\sum_{k}^{n} u_k + x}  \right)}
    {\left(\sum_{k}^{n} u_k + x \right)
    \left( \sum_k^m y_k + \frac{\sum_{k}^{n} u_k x}{ \sum_{k}^{n} u_k + x} \right)} 
   = \frac{\frac{u_i u_j}{\sum_{k}^{n} u_k + x } 
    \left( \sum_k^m y_k + \frac{\sum_{k}^{n} u_k x}
    {\sum_{k}^{n} u_k + x} +  \frac{x^2}{\sum_{k}^{n} u_k + x} \right) }
    {\left( \sum_k^m y_k + \frac{\sum_{k}^{n} u_k x}{ \sum_{k}^{n} u_k + x} \right)} 
\end{align}

The denominators of equations \ref{eqn: LHS} and \ref{eqn: RHS} are the same. Hence, multiplying the numerators of these equations by $\frac{\sum_{k}^{n} u_k + x }{u_i u_j}$ , it suffices to prove that: 
\begin{align} 
\sum_k^m y_k + \frac{\sum_{k}^{n} u_k x}{\sum_{k}^{n} u_k + x} +  \frac{x^2}{\sum_{k}^{n} u_k + x} = \sum_k^m y_k + x \left( 
\frac{\sum_{k}^{n} u_k + x}{\sum_{k}^{n} u_k + x} \right) = \sum_k^m y_k + x
\end{align}

The proof for equation \ref{eqn: confluence-2} is analogous. 

\end{proof}

The following is the main result of this paper:

\begin{theorem}
${\sf Resist}_{R}$ has a confluent terminating rewriting system on maps.
\end{theorem}

\begin{proof}

First observe that the reduction using rewriting rules of ${\sf Resist}_R$ must terminate. This may be observed by keeping track of the number of nodes $N$ (after expansion) and the number of parallel arrows $P$ and using lexicographical ordering on the pairs $(N,P)$.  The $[{\sf Spider}]$ and $[{\sf Parallel}]$ rules reduce both $N$ and $P$. The star-mesh family of identities reduce the number of nodes $N$ as an internal node is always removed -- note that the star-mesh rule, on the other hand, can increase the number of parallel connections. As the lexicographical ordering on $\mathbb{N} \times \mathbb{N}$ is a well-ordering, this shows that the reduction process must eventually terminate.

Since the rewriting terminates, local confluence, that is {\em resolutions} of diverging single step rewrites, implies global confluence. Hence, it suffices to prove the local confluence property for the distinct pairs of terms produced by overlapping divergent one step rewrites: these are often called critical pairs. This amounts to proving that the order of reducing overlapping rewrites does not matter. 
The critical pairs (excluding the spider rewrites) that occur in the this rewriting system are drawn below:

\begin{enumerate}[(a)]
    \item {\bf Overlapping [{\sf Parallel}] rewrites:}
    
    \begin{minipage}[h]{0.3\textwidth}
            \[ \includegraphics[scale=0.07]{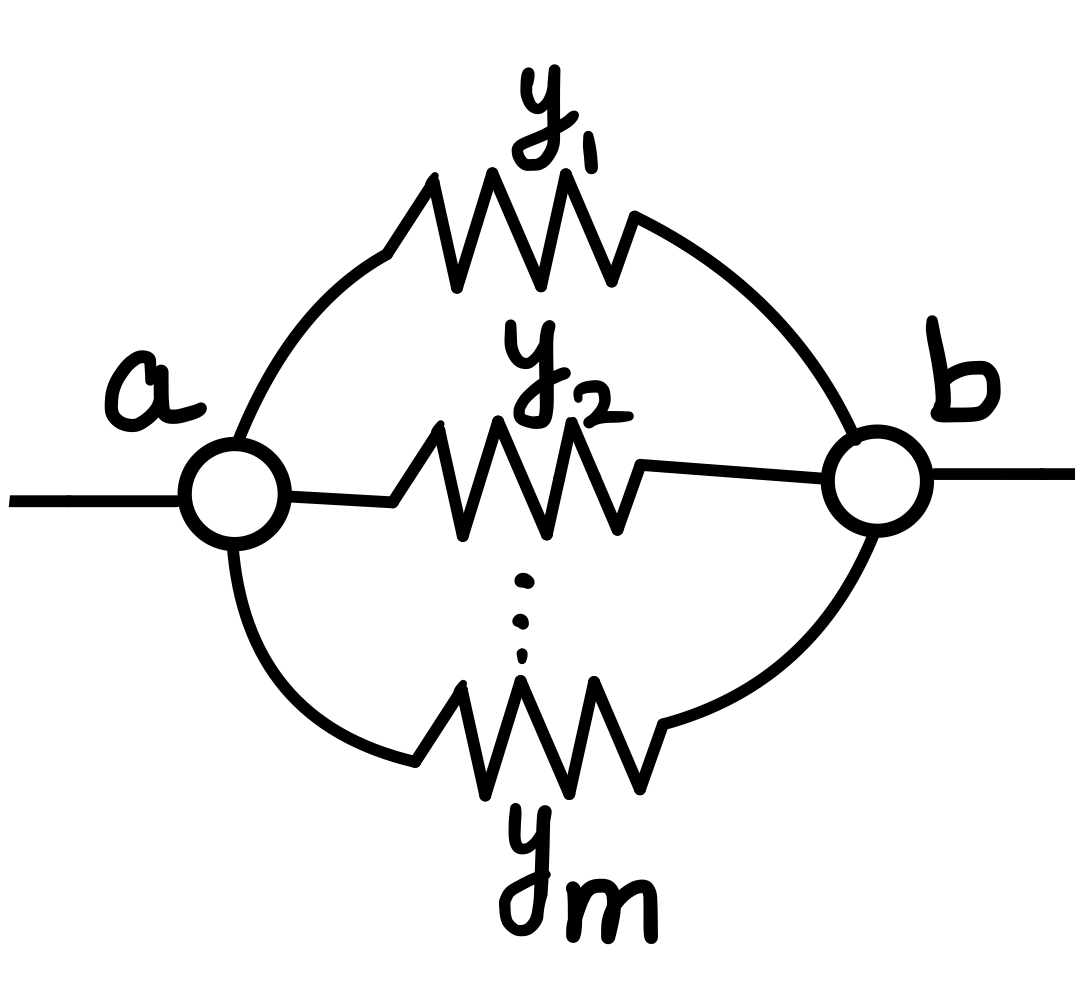} \]
    \end{minipage}
    \begin{minipage}[h]{0.6\textwidth}
    Rewriting the network on the left requires $m-1$ overlapping applications of the [{\sf Parallel}] rule. Since, combining two parallel resistors involves adding their conductances, and addition is associative, all the orders of application of the [{\sf Parallel}] rule yield the same final circuit. Hence, local confluence holds in this case. 

    \end{minipage}
    \item{\bf Overlapping $(Y/\Delta)_1$ and $(Y/\Delta)_2$ rewrite:}
    
    \begin{minipage}[h]{0.3\textwidth}
            \[ \includegraphics[scale=0.035]{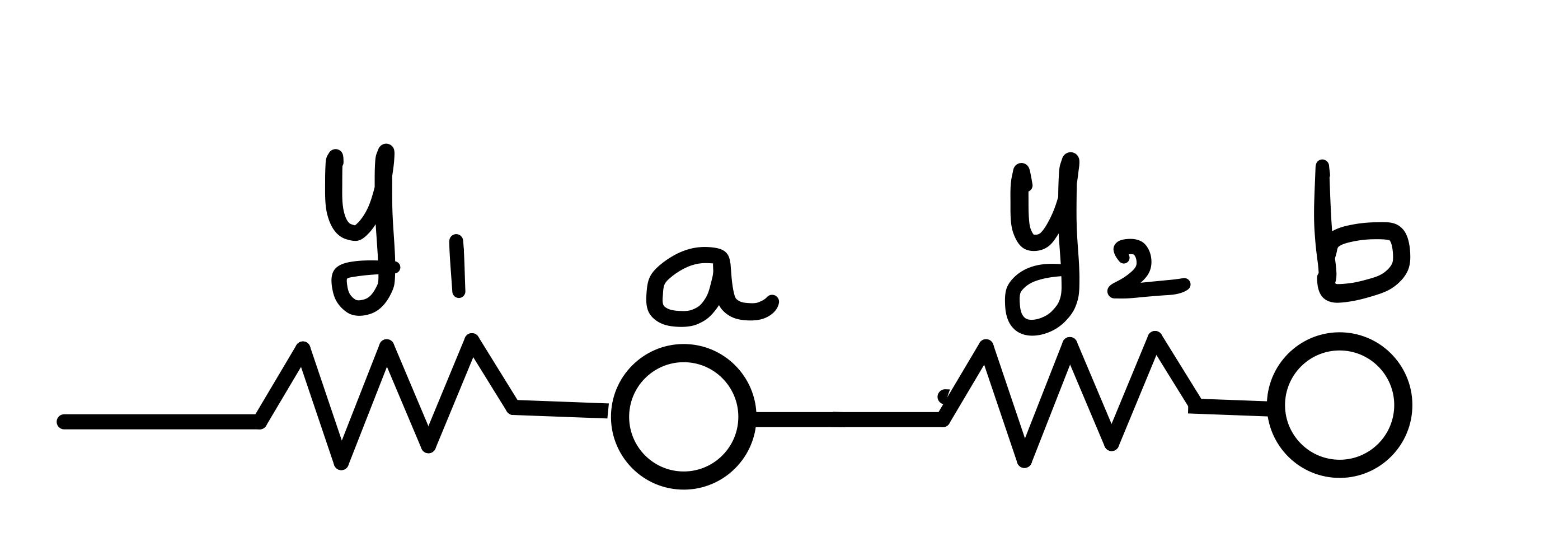} \]
    \end{minipage}
    \begin{minipage}[h]{0.6\textwidth}
    To rewrite the network on the left, one may apply $(Y/\Delta)_1$ to eliminate node $b$ first, or apply $(Y/\Delta)_2$ to eliminate node $a$ first. This results in a critical pair $((Y/\Delta)_1, (Y/\Delta)_2)$ in the rewriting process of such networks. However, the critical pair is locally confluent, see below: 
    \end{minipage}
    \[ \includegraphics[scale=0.03]{diag/CP1.png} 
   ~~\xRightarrow{(Y\Delta)_2} ~~\includegraphics[scale=0.03]{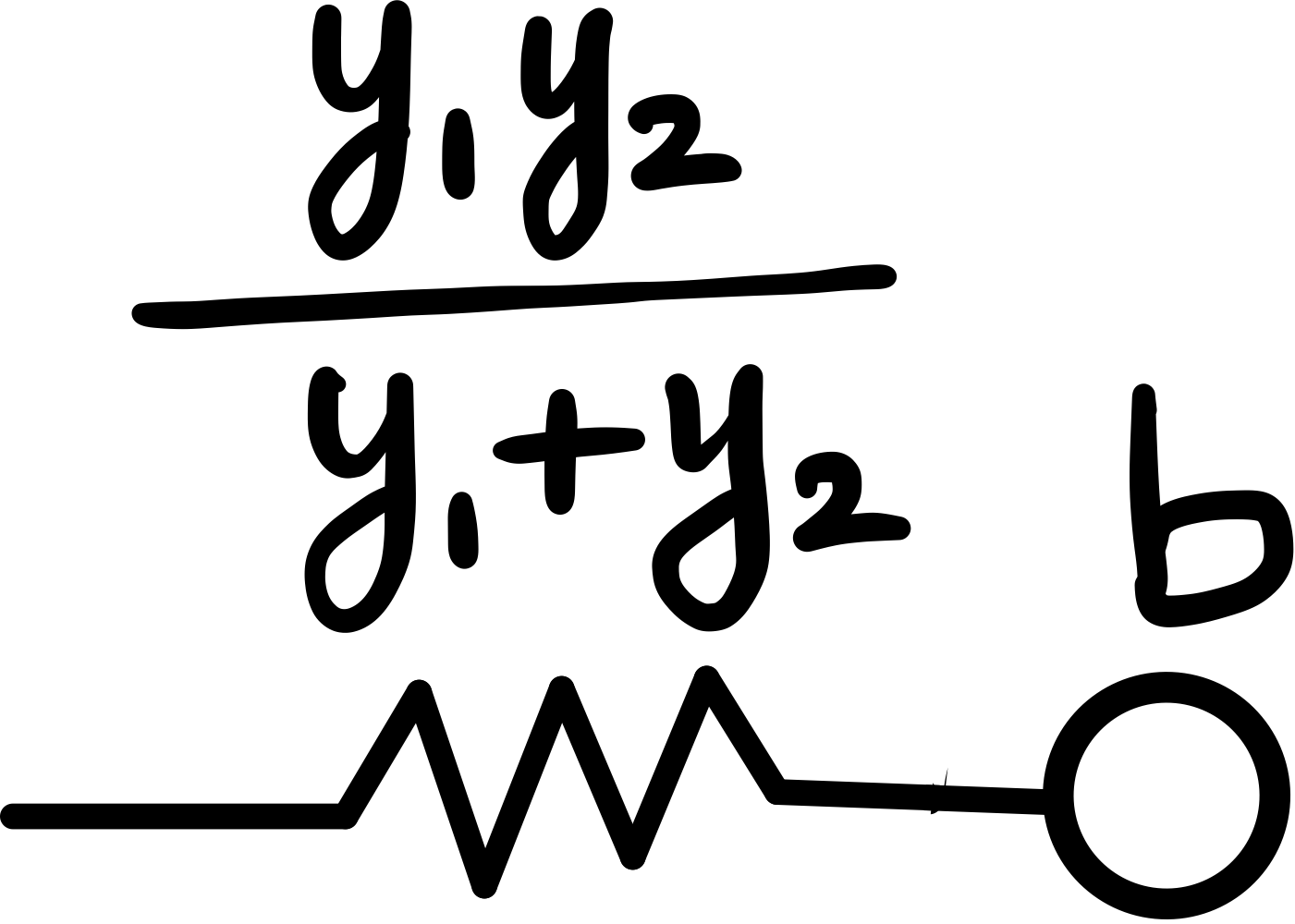} 
   ~~\xRightarrow{(Y\Delta)_1} ~~\includegraphics[scale=0.015]{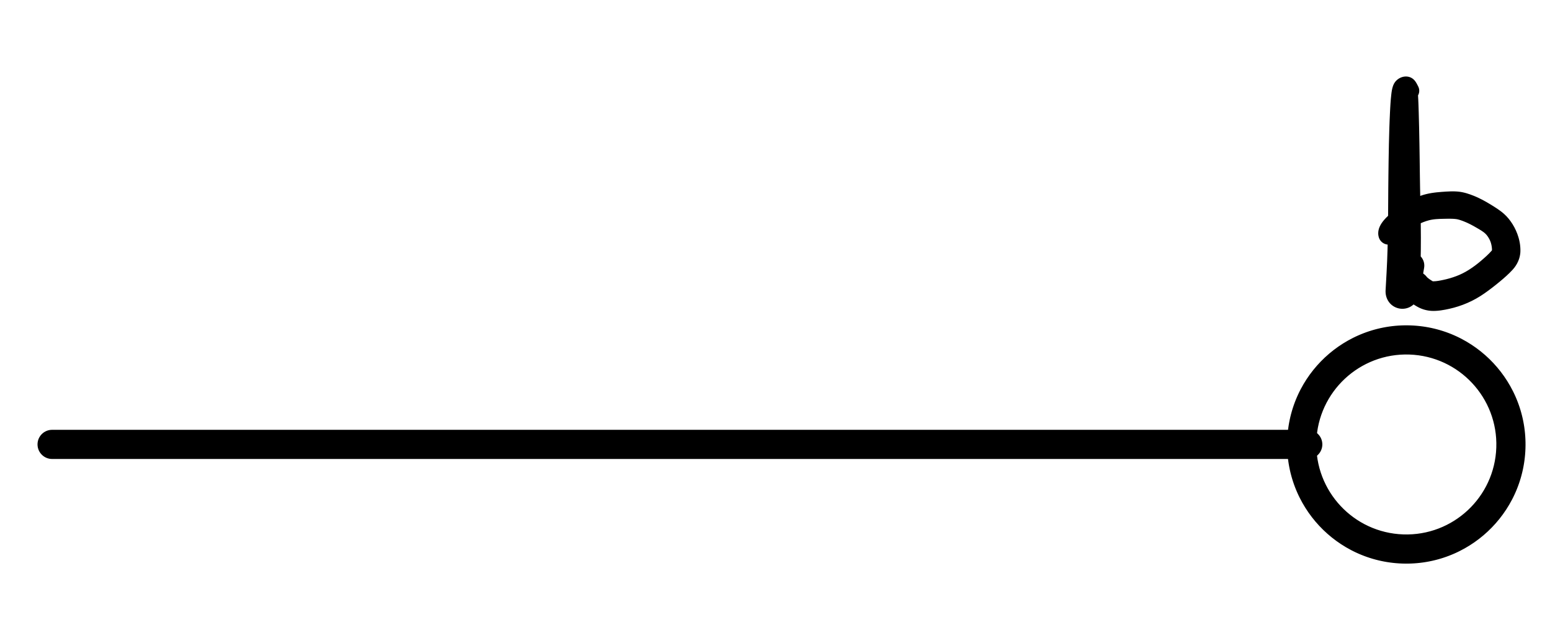} \]
   \[ \includegraphics[scale=0.03]{diag/CP1.png} 
      ~~ \xRightarrow{(Y\Delta)_1} ~~ \includegraphics[scale=0.026]{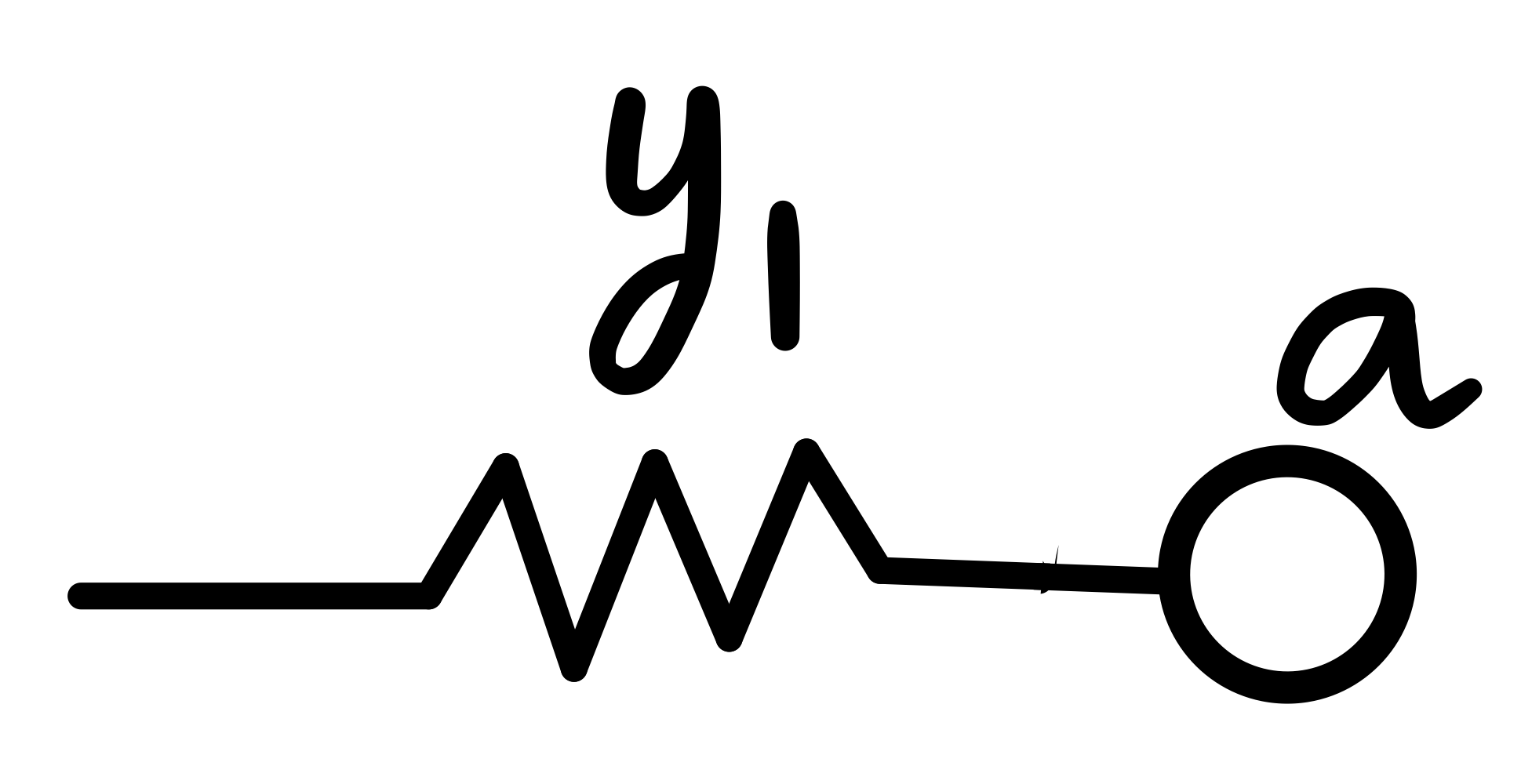} 
      ~~ \xRightarrow{(Y\Delta)_1} ~~ \includegraphics[scale=0.015]{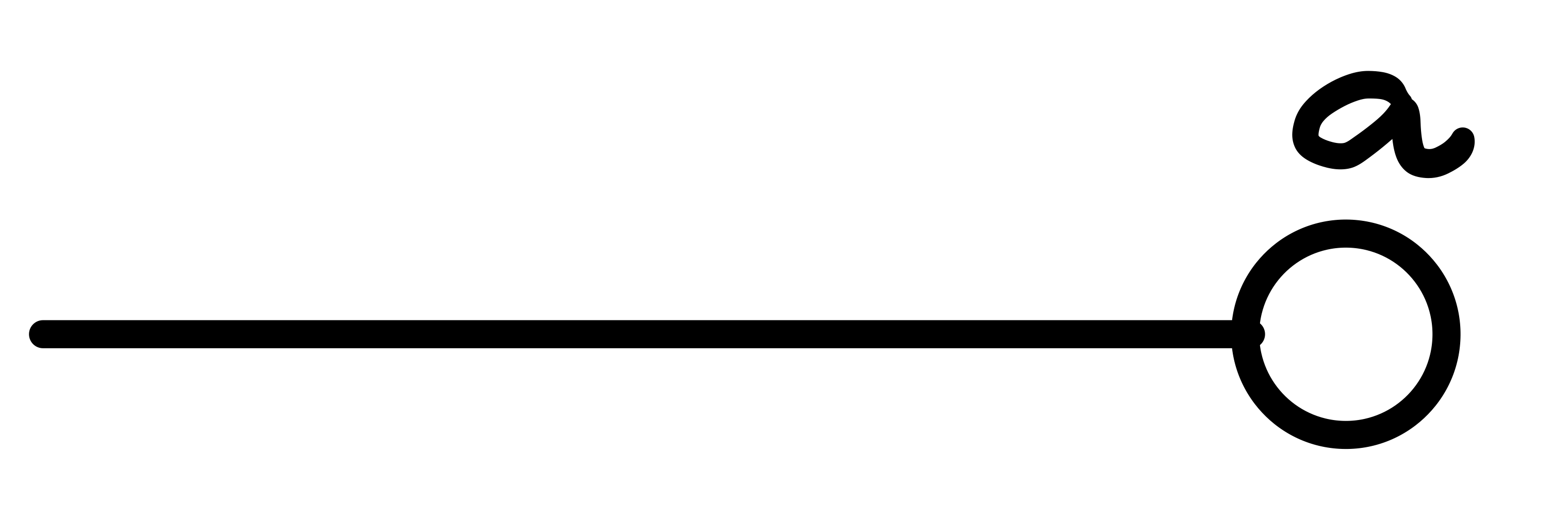} \]
        
    \item {\bf Overlapping $(Y/\Delta)_1$ and star-mesh rewrite:}

    \begin{minipage}[h]{0.3\textwidth}
            \[ \includegraphics[scale=0.05]{diag/Unit1.png} \]
    \end{minipage}
    \begin{minipage}[h]{0.6\textwidth}
             To rewrite the network on the left, one may apply $(Y/\Delta)_1$ to eliminate node $b$ first (and then apply $(Y/\Delta)_m$ to eliminate $a$), or apply $(Y/\Delta)_{m+1}$ to eliminate node $a$ first (and then apply $(Y/\Delta)_m$ to eliminate $b$). This results in a critical pair $((Y/\Delta)_1,(Y/\Delta)_m) $ in the rewriting of such networks. However, by Lemma \ref{Lemma: unit-star rewrite}, local confluence holds for this critical pair. 
    \end{minipage}

    \item {\bf Overlapping $\sf [Parallel]$ and star-mesh rewrites:}
    
    \begin{minipage}[h]{0.3\textwidth}
            \[ \includegraphics[scale=0.05]{diag/N9.png} \]
    \end{minipage}
    \begin{minipage}[h]{0.6\textwidth}
    To rewrite the network on the left, ${\sf [Parallel]}$ to combine resistors $y_1$ and $y_2$ first (followed by $(Y/\Delta)_{n-1}$ to eliminate node $x$), or apply $(Y/\Delta)_{n}$ may be applied to eliminate node $x$ first. This results in a critical pair $({\sf [Parallel]}, (Y/\Delta)_{n})$ in the rewriting process of such networks. However, by Lemma \ref{Lemma: star-parallel rewrite}, local confluence holds for this critical pair. 
    \end{minipage}    
    
    \item {\bf Overlapping two star-mesh rewrites:}
    
    \begin{minipage}[h]{0.3\textwidth}
            \[ \includegraphics[scale=0.045]{diag/Confluence_2.png} \]
    \end{minipage}
    \begin{minipage}[h]{0.6\textwidth}
    To rewrite the network on the left, $(Y/\Delta)_{n+1}$ may be applied first to eliminate node $b$ first (followed by $(Y/\Delta)_{m+n}$ to eliminate node $a$), or $(Y/\Delta)_{m+1}$ may be applied to eliminate node $a$ first (followed by $(Y/\Delta)_{m+n}$ to eliminate node $b$). This results in a critical pair $((Y/\Delta)_{n+1}, (Y/\Delta)_{m+1})$ in the rewriting process of such networks. However, by Lemma \ref{Lemma: star-star rewrite}, local confluence holds for this critical pair. 
    \end{minipage}
\end{enumerate}

\end{proof}

An immediate consequence is:
\begin{corollary}
Modulo the decidability of the positive division rig $R$, $
{\sf Resist}_{R} $ has a decidable equality by reduction to normal form.
\end{corollary}

\section{Discussion}
In this paper we have provided a normal form for resistor networks over a positive division rig and, thereby, a decision procedure for equality of resistor circuits (given decidability of equality for the rig).  Even though modest, as far as we know, ours is the first such result in the literature.

Of course, our motivation came from the difficulty of working with the existing `normal forms' for stabilizer circuits 
\cite{backens2016completeness}. As pointed out by Kissinger in \cite{kissinger2022phase}, these normal forms for stabilizer circuits are hard to work with and {\em ``almost a decade after completeness was proven for the stabiliser fragment of the ZX calculus, new ideas are still needed"}.  Based on the recently established connection between quantum and electrical circuits \cite{comfort2021graphical, cockett2022categories}, our thought was that, studying simpler cases such as resistor circuits might provide new insights into normal forms for stabilizer circuits.

An interesting and a more challenging question is whether the results in this paper can be generalized to arbitrary division rigs so as to cover the `resistor' case arising from qudit stabilizer quantum mechanics. As has been mentioned, to apply these ideas to stabilizer circuits a necessary step is to generalize these results to division rigs and so, in particular, to finite fields.  The technical difficulty of applying these ideas verbatim is the question of how one handles zeros and divisions by zero. 

An example of this difficulty over a finite field, arises when resolving the critical pair $(Y/\Delta)_1$ with 
$(Y/\Delta)_n$ (for $n \geq 3$)  (essentially Lemma \ref{Lemma: unit-star rewrite} above).  The rewriting of the $n$-star 
to an $n$-mesh can involve a division by zero (when $\sum_{i = 1}^n y_i = 0$): it is tempting to think that this should be 
interpreted as giving ``infinite conductances'' in the mesh.  However, removing a point of the star using $(Y/\Delta)_1$ 
will also remove the division by zero in the subsequent $(Y/\Delta)_{n-1}$ rewriting showing such an interpretation is not 
valid.  

The point is that for finite fields and division rings the star/mesh transformations are only valid when the sum of the 
conductances of the star is non-zero.   
This, of course complicates the rewriting story and also reopens the question what a convenient presentation of resistor circuits over finite fields might be! We leave resolving these issues for future work.
 
A surprisingly basic -- and as far as we know open -- question which arises from this work concerns whether there is a finite presentation of ${\sf Resist}_R$ in terms of generators and relation. In order to provide a presentation for ${\sf Resist}_R$, we assume an infinite family of star-mesh identities, $(Y/\Delta)_n$, for each $n \in \mathbb{N}$.  While we have shown that this infinite set of identities completely characterize equality between circuits, it is an open question whether there is a finite presentation of the category. We conjecture that there isn't one.
\section{Acknowledgements}
ARK would like to thank NTT Research for financial and technical support. Research at IQC is supported in part by the Government of Canada through Innovation, Science and Economic Development Canada (ISED).
\newpage

\bibliographystyle{eptcs}
\bibliography{resistor}

\begin{thebibliography}{10}
\providecommand{\bibitemdeclare}[2]{}
\providecommand{\surnamestart}{}
\providecommand{\surnameend}{}
\providecommand{\urlprefix}{Available at }
\providecommand{\url}[1]{\texttt{#1}}
\providecommand{\href}[2]{\texttt{#2}}
\providecommand{\urlalt}[2]{\href{#1}{#2}}
\providecommand{\doi}[1]{doi:\urlalt{https://doi.org/#1}{#1}}
\providecommand{\eprint}[1]{arXiv:\urlalt{https://arxiv.org/abs/#1}{#1}}
\providecommand{\bibinfo}[2]{#2}

\bibitemdeclare{phdthesis}{backens2016completeness}
\bibitem{backens2016completeness}
\bibinfo{author}{Miriam \surnamestart Backens\surnameend}
  (\bibinfo{year}{2016}): \emph{\bibinfo{title}{Completeness and the
  ZX-calculus}}.
\newblock Ph.D. thesis, \bibinfo{school}{Oxford University},
  \doi{10.48550/arXiv.1602.08954}.
\newblock \bibinfo{note}{ArXiv preprint arXiv:1602.08954}.

\bibitemdeclare{unpublished}{Baeztalk}
\bibitem{Baeztalk}
\bibinfo{author}{John \surnamestart Baez\surnameend} \&
  \bibinfo{author}{Brendan \surnamestart Fong\surnameend}
  (\bibinfo{year}{2015}): \emph{\bibinfo{title}{Circuits, Categories and
  Rewrite rules}}.
\newblock
  \urlprefix\url{https://math.ucr.edu/home/baez/networks_warsaw/circuits_web_warsaw.pdf}.
\newblock \bibinfo{note}{Higher-Dimensional Rewriting and Applications, Warsaw,
  June 2015}.

\bibitemdeclare{article}{baez2015compositional}
\bibitem{baez2015compositional}
\bibinfo{author}{John~C \surnamestart Baez\surnameend} \&
  \bibinfo{author}{Brendan \surnamestart Fong\surnameend}
  (\bibinfo{year}{2015}): \emph{\bibinfo{title}{A compositional framework for
  passive linear networks}}.
\newblock {\slshape \bibinfo{journal}{arXiv preprint arXiv:1504.05625}},
  \doi{10.48550/arXiv.1504.05625}.

\bibitemdeclare{inproceedings}{zanasi}
\bibitem{zanasi}
\bibinfo{author}{Filippo \surnamestart Bonchi\surnameend},
  \bibinfo{author}{Robin \surnamestart Piedeleu\surnameend},
  \bibinfo{author}{Pawel \surnamestart Sobociński\surnameend} \&
  \bibinfo{author}{Fabio \surnamestart Zanasi\surnameend}
  (\bibinfo{year}{2019}): \emph{\bibinfo{title}{Graphical Affine Algebra}}.
\newblock In: {\slshape \bibinfo{booktitle}{2019 34th Annual ACM/IEEE Symposium
  on Logic in Computer Science (LICS)}}, pp. \bibinfo{pages}{1--12},
  \doi{10.1109/LICS.2019.8785877}.

\bibitemdeclare{article}{cockett2022categories}
\bibitem{cockett2022categories}
\bibinfo{author}{Robin \surnamestart Cockett\surnameend},
  \bibinfo{author}{Amolak~Ratan \surnamestart Kalra\surnameend} \&
  \bibinfo{author}{Shiroman \surnamestart Prakash\surnameend}
  (\bibinfo{year}{2022}): \emph{\bibinfo{title}{Categories of Kirchhoff
  relations}}.
\newblock {\slshape \bibinfo{journal}{arXiv preprint arXiv:2205.05870}},
  \doi{10.48550/arXiv.2205.05870}.

\bibitemdeclare{article}{comfort2021graphical}
\bibitem{comfort2021graphical}
\bibinfo{author}{Cole \surnamestart Comfort\surnameend} \&
  \bibinfo{author}{Aleks \surnamestart Kissinger\surnameend}
  (\bibinfo{year}{2021}): \emph{\bibinfo{title}{A graphical calculus for
  Lagrangian relations}}.
\newblock {\slshape \bibinfo{journal}{arXiv preprint arXiv:2105.06244}},
  \doi{10.4204/EPTCS.372.24}.

\bibitemdeclare{misc}{Coyer}
\bibitem{Coyer}
\bibinfo{author}{Brandon \surnamestart Coya\surnameend} (\bibinfo{year}{2018}):
  \emph{\bibinfo{title}{Circuits, Bond Graphs, and Signal-Flow Diagrams: A
  Categorical Perspective}}, \doi{10.48550/ARXIV.1805.08290}.
\newblock \urlprefix\url{https://arxiv.org/abs/1805.08290}.

\bibitemdeclare{article}{fong2016algebra}
\bibitem{fong2016algebra}
\bibinfo{author}{Brendan \surnamestart Fong\surnameend} (\bibinfo{year}{2016}):
  \emph{\bibinfo{title}{The algebra of open and interconnected systems}}.
\newblock {\slshape \bibinfo{journal}{arXiv preprint arXiv:1609.05382}},
  \doi{10.48550/arXiv.1609.05382}.

\bibitemdeclare{article}{fong2019hypergraph}
\bibitem{fong2019hypergraph}
\bibinfo{author}{Brendan \surnamestart Fong\surnameend} \&
  \bibinfo{author}{David~I \surnamestart Spivak\surnameend}
  (\bibinfo{year}{2019}): \emph{\bibinfo{title}{Hypergraph categories}}.
\newblock {\slshape \bibinfo{journal}{Journal of Pure and Applied Algebra}}
  \bibinfo{volume}{223}(\bibinfo{number}{11}), pp. \bibinfo{pages}{4746--4777},
  \doi{10.48550/arXiv.1806.08304}.

\bibitemdeclare{article}{PhysRevB.37.302}
\bibitem{PhysRevB.37.302}
\bibinfo{author}{D.~J. \surnamestart Frank\surnameend} \&
  \bibinfo{author}{C.~J. \surnamestart Lobb\surnameend} (\bibinfo{year}{1988}):
  \emph{\bibinfo{title}{Highly efficient algorithm for percolative transport
  studies in two dimensions}}.
\newblock {\slshape \bibinfo{journal}{Phys. Rev. B}} \bibinfo{volume}{37}, pp.
  \bibinfo{pages}{302--307}, \doi{10.1103/PhysRevB.37.302}.
\newblock \urlprefix\url{https://link.aps.org/doi/10.1103/PhysRevB.37.302}.

\bibitemdeclare{article}{Geo88}
\bibitem{Geo88}
\bibinfo{author}{Ghassan \surnamestart George~Batrouni\surnameend} \&
  \bibinfo{author}{Alex \surnamestart Hansen\surnameend}
  (\bibinfo{year}{1988}): \emph{\bibinfo{title}{Fourier acceleration of
  iterative processes in disordered systems}}.
\newblock {\slshape \bibinfo{journal}{Journal of statistical physics}}
  \bibinfo{volume}{52}, pp. \bibinfo{pages}{747--773},
  \doi{10.1007/BF01019728}.

\bibitemdeclare{mastersthesis}{kalra2022category}
\bibitem{kalra2022category}
\bibinfo{author}{Amolak~Ratan \surnamestart Kalra\surnameend}
  (\bibinfo{year}{2022}): \emph{\bibinfo{title}{The Category of Kirchhoff
  Relations}}.
\newblock Master's thesis, \bibinfo{school}{University of Calgary},
  \doi{10.11575/PRISM/39953}.

\bibitemdeclare{article}{kissinger2022phase}
\bibitem{kissinger2022phase}
\bibinfo{author}{Aleks \surnamestart Kissinger\surnameend}
  (\bibinfo{year}{2022}): \emph{\bibinfo{title}{Phase-free ZX diagrams are CSS
  codes (... or how to graphically grok the surface code)}}.
\newblock {\slshape \bibinfo{journal}{arXiv preprint arXiv:2204.14038}},
  \doi{10.48550/arXiv.2204.14038}.

\bibitemdeclare{article}{Knud06}
\bibitem{Knud06}
\bibinfo{author}{Henning~Arendt \surnamestart Knudsen\surnameend} \&
  \bibinfo{author}{S{\'a}ndor \surnamestart Fazekas\surnameend}
  (\bibinfo{year}{2006}): \emph{\bibinfo{title}{Robust algorithm for random
  resistor networks using hierarchical domain structure}}.
\newblock {\slshape \bibinfo{journal}{Journal of computational physics}}
  \bibinfo{volume}{211}(\bibinfo{number}{2}), pp. \bibinfo{pages}{700--718},
  \doi{10.1016/j.jcp.2005.06.007}.

\bibitemdeclare{article}{Romm09}
\bibitem{Romm09}
\bibinfo{author}{Joost \surnamestart Rommes\surnameend} \&
  \bibinfo{author}{Wil~HA \surnamestart Schilders\surnameend}
  (\bibinfo{year}{2009}): \emph{\bibinfo{title}{Efficient methods for large
  resistor networks}}.
\newblock {\slshape \bibinfo{journal}{IEEE Transactions on Computer-Aided
  Design of Integrated Circuits and Systems}}
  \bibinfo{volume}{29}(\bibinfo{number}{1}), pp. \bibinfo{pages}{28--39},
  \doi{10.1109/TCAD.2009.2034402}.

\bibitemdeclare{article}{stardelta}
\bibitem{stardelta}
\bibinfo{author}{D.~W.~C. \surnamestart Shew\surnameend}
  (\bibinfo{year}{1947}): \emph{\bibinfo{title}{XXVII. Generalized star and
  mesh transformations}}.
\newblock {\slshape \bibinfo{journal}{The London, Edinburgh, and Dublin
  Philosophical Magazine and Journal of Science}}
  \bibinfo{volume}{38}(\bibinfo{number}{279}), pp. \bibinfo{pages}{267--275},
  \doi{10.1080/14786444708521594}.

\end{thebibliography}

\appendix
\newpage
\section{Hypergraph categories}
\label{Appendix: hypergraph}

A category is a {\bf hypergraph category} in case it is a symmetric monoidal category in  which every object is coherently a special commutative Fr\"obenius algebra, \cite{fong2019hypergraph}.  
This means that each object $X$ in the category  has an associated special Fr\"obenius algebra structure $(X, \nabla_X: X \ox X \to X, \eta_X: I \to X, \Delta_X: X \to X \ox X, \epsilon_X: X \to I)$, whose identities are graphically depicted below (with $\circ$ indicating both the multiplication, comultiplication, and units): 

The comultiplication, the counit, the multiplication and the unit are drawn as follows:
\[ \includegraphics[scale=0.03]{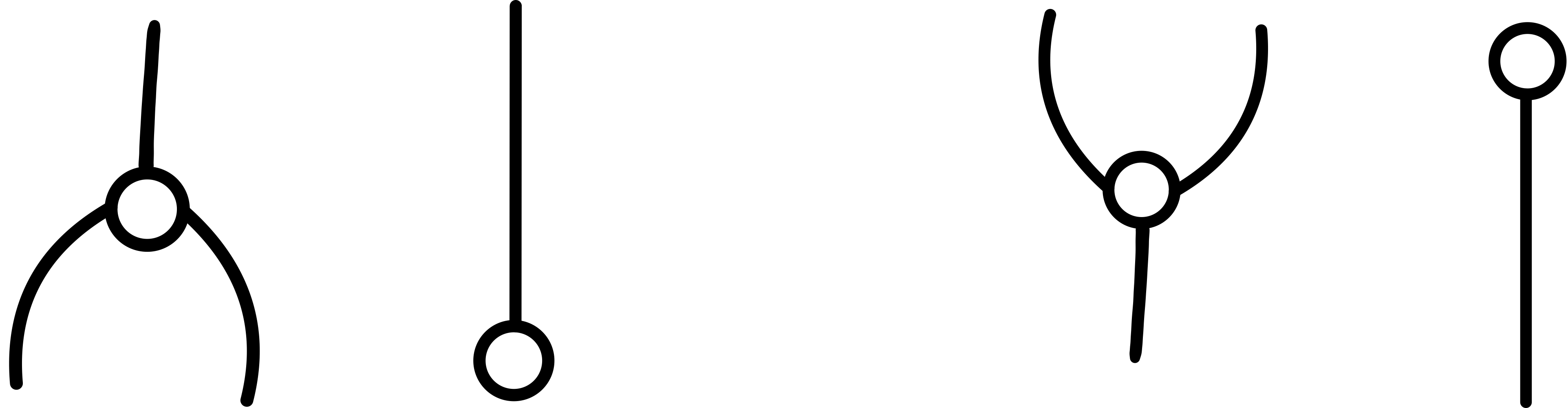} \] 
The maps satisfy the following equations and their vertically flipped image:
\begin{align*} 
(a) \vcenteredinclude{scale=0.07}{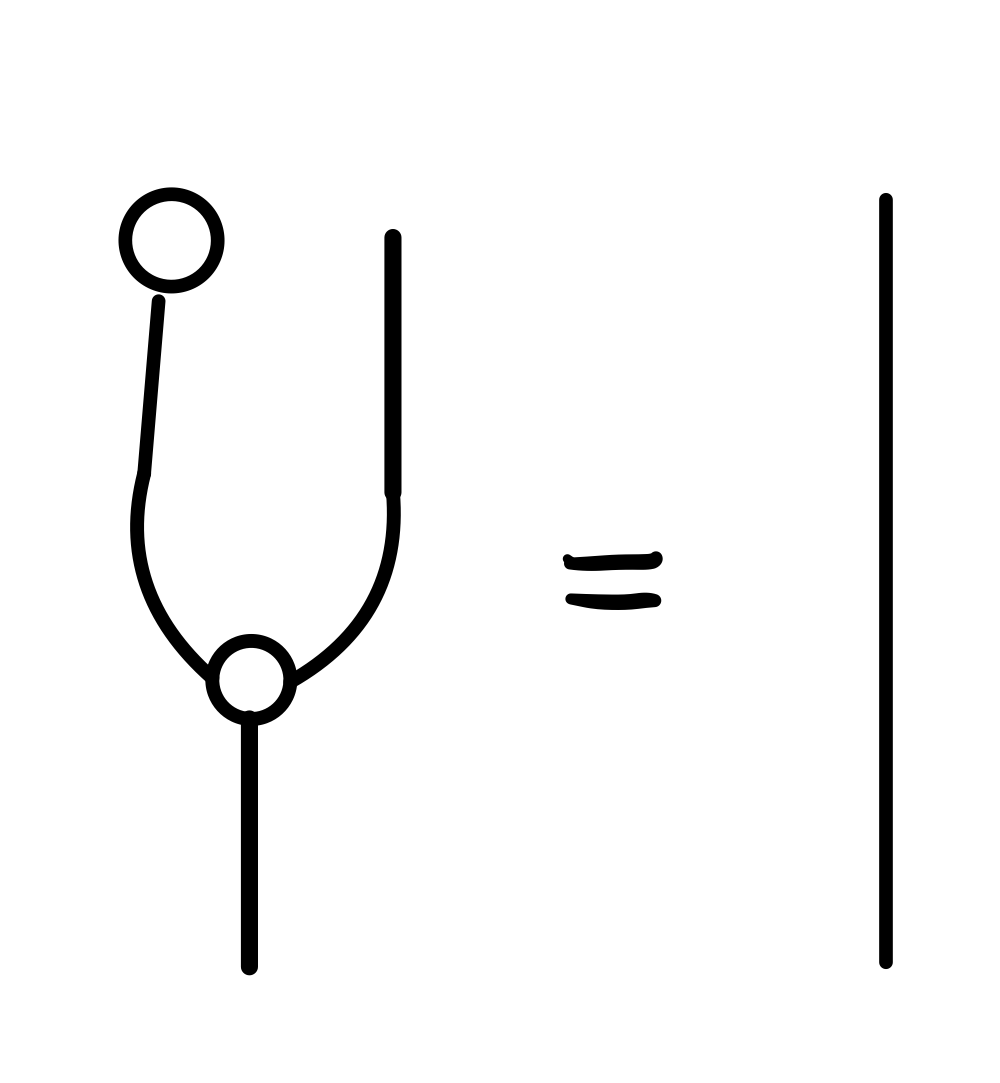} ~~~~~ 
(b) \vcenteredinclude{scale=0.07}{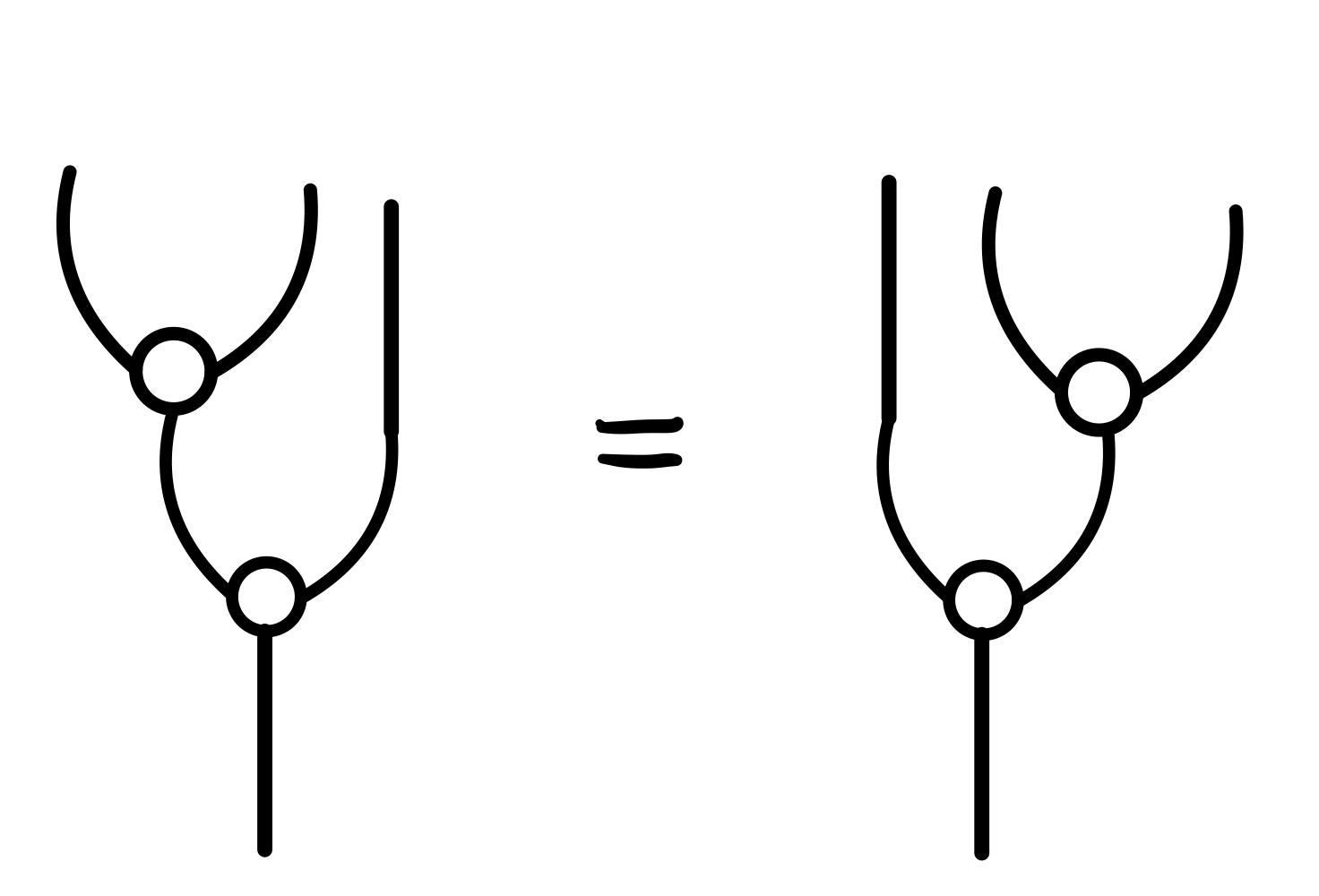} ~~~~~
(c) \vcenteredinclude{scale=0.06}{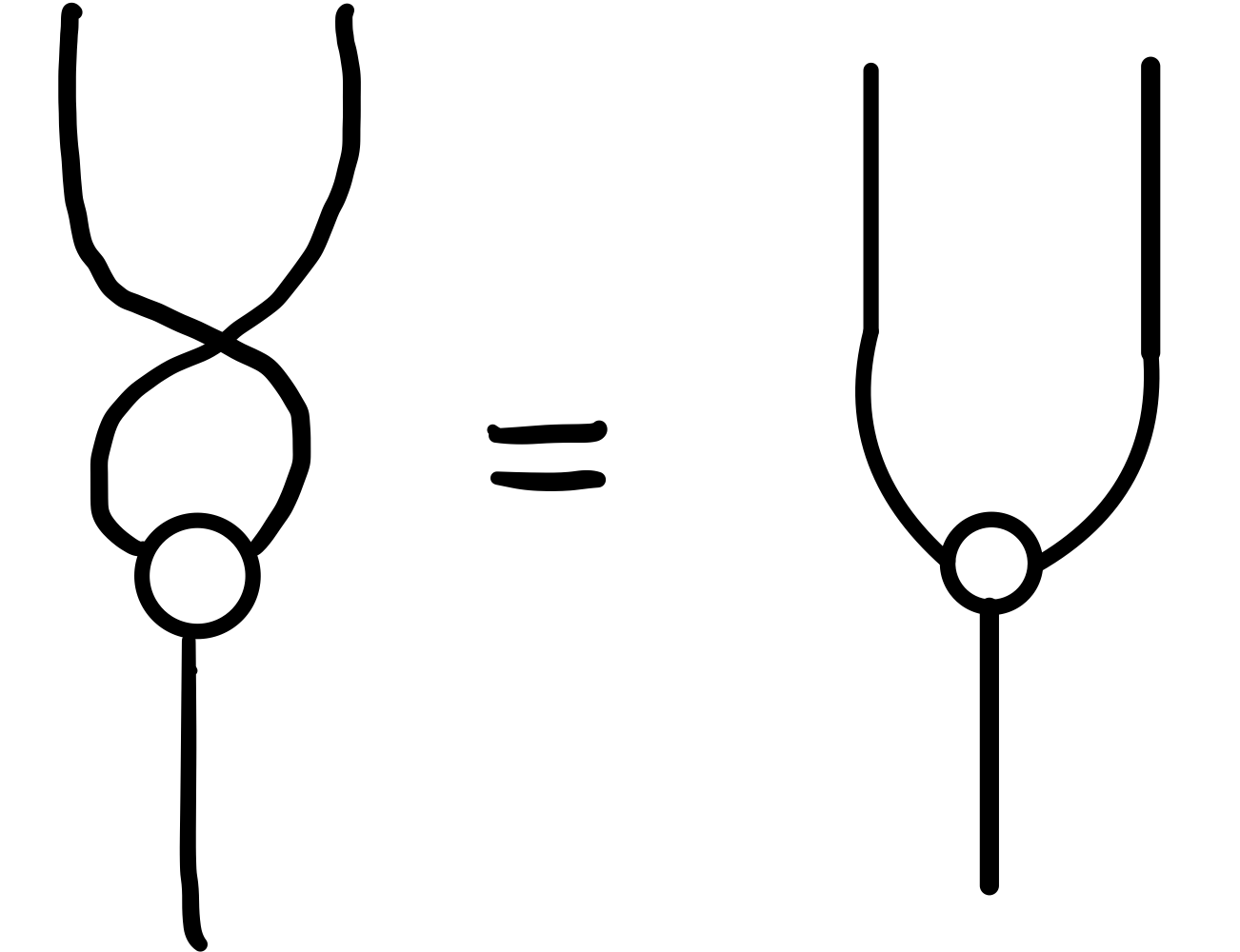} ~~~~~
\end{align*}
\begin{align} \label{eq:Frob}
 (e)~~ \vcenteredinclude{scale=0.07}{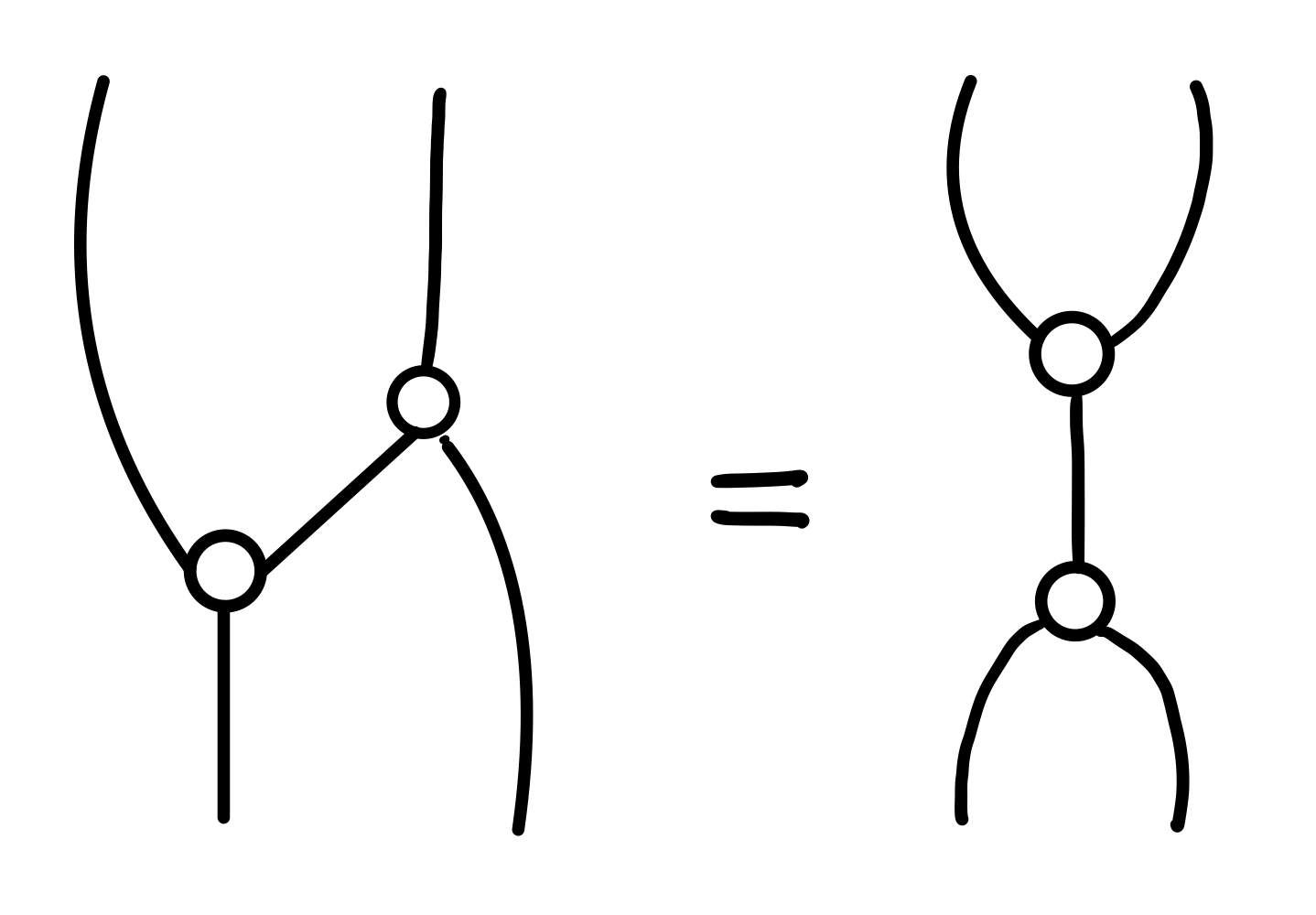} ~~~~~ 
 (f)~~ \vcenteredinclude{scale=0.07}{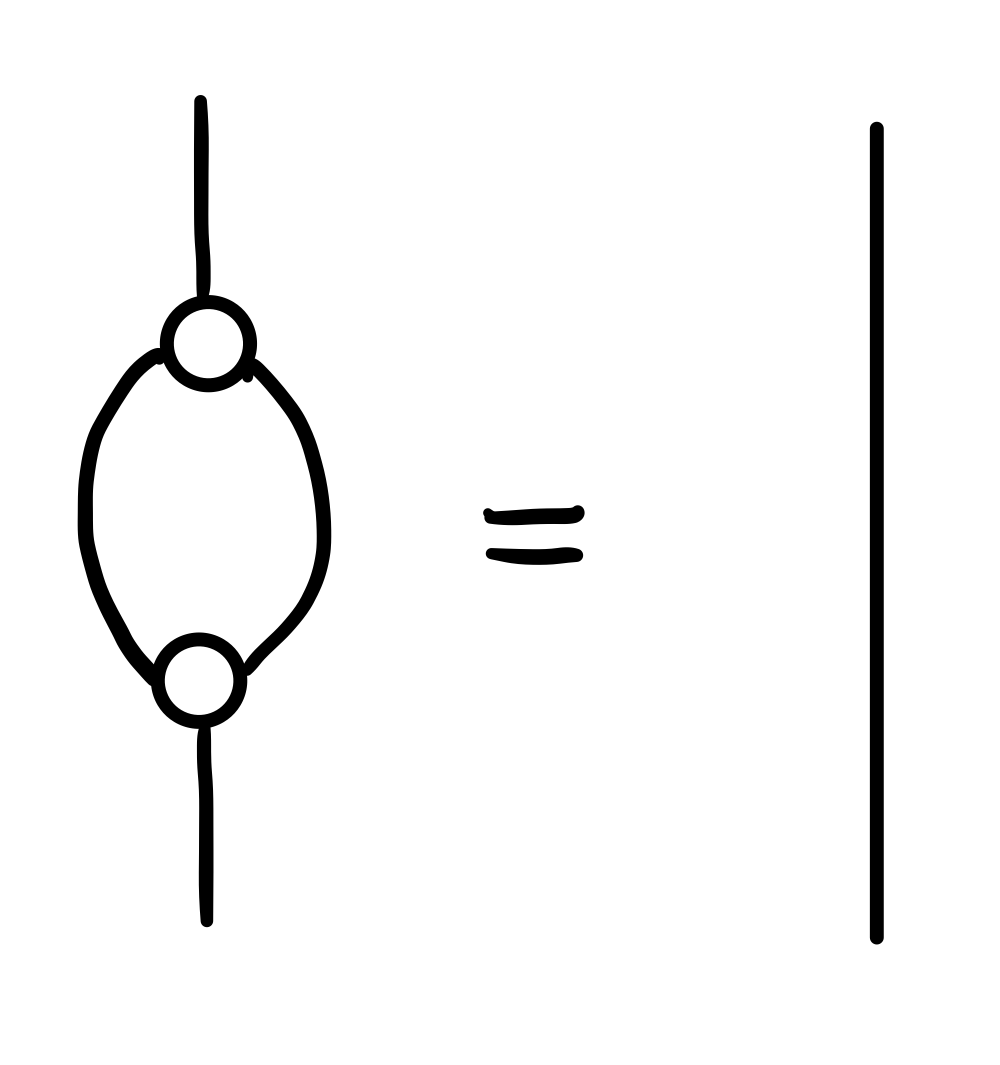} 
\end{align}

The Frobenius structure is not natural but is ``coherent'' in the sense that the multiplication on the tensor of two objects is given by $\Delta_{X \ox Y}:= (X \ox Y) \ox (X \ox Y) \to^{\sf ex}_{\simeq} (X \ox X) \ox (Y \ox Y) \to^{\Delta_X \ox \Delta_Y} X \ox Y$ and the comultiplication and units are similarly given.

Hypergraph categories are automatically compact closed: each object is self-dual.   This has the effect that the directionality of inputs and outputs is not as important as the connectivity.   

\end{document}